\documentclass[
final
]{dmtcs-episciences}

\usepackage[utf8]{inputenc}
\usepackage{subfigure}

\usepackage[round]{natbib}

\usepackage{tikz,color,url}
\usepackage{amsmath,amssymb,latexsym}
\usetikzlibrary{arrows}

\newtheorem{theorem}{Theorem}[section]

\newtheorem{lemma}[theorem]{Lemma}
\newtheorem{proposition}[theorem]{Proposition}

\newtheorem{p}{Problem}
\newtheorem{con}{Conjecture}
\newtheorem{exa}[theorem]{Example}

\DeclareMathOperator{\TC}{TC}
\DeclareMathOperator{\DC}{DC}
\DeclareMathOperator{\DCG}{DCG}

\author[Bo\v{s}tjan Bre\v{s}ar et al.]{Bo\v{s}tjan Bre\v{s}ar\affiliationmark{1,2,}\thanks{Supported by the Slovenian Research and Innovation Agency (ARIS) under the grants P1-0297, N1-0285, and N1-0355.}
  \and Sandi Klav\v zar\affiliationmark{1,2,3,}\thanks{Supported by the Slovenian Research and Innovation Agency (ARIS) under the grants P1-0297, N1-0285, and N1-0355.}
  \and Babak Samadi\affiliationmark{2}
}
\title[Total $k$-coalition]{Total $k$-coalition: bounds, exact values and an application to double coalition}
\affiliation{
  Faculty of Natural Sciences and Mathematics, University of Maribor, Maribor, Slovenia\\
  Institute of Mathematics, Physics and Mechanics, Ljubljana, Slovenia\\
  Faculty of Mathematics and Physics, University of Ljubljana, Ljubljana, Slovenia}
\begin{document}
\publicationdata{vol. 27:3}{2025}{3}{10.46298/dmtcs.15231}{2025-02-12; 2025-02-12; 2025-06-24}{2025-06-24}
\maketitle

\begin{abstract}
Let $G=\big{(}V(G),E(G)\big{)}$ be a graph with minimum degree $k$. A subset $S\subseteq V(G)$ is called a total $k$-dominating set if every vertex in $G$ has at least $k$ neighbors in $S$. Two disjoint sets $A,B\subset V(G)$ form a total $k$-coalition in $G$ if none of them is a total $k$-dominating set in $G$ but their union $A\cup B$ is a total $k$-dominating set. A vertex partition $\Omega=\{V_{1},\ldots,V_{|\Omega|}\}$ of $G$ is a total $k$-coalition partition if each set $V_{i}$ forms a total $k$-coalition with another set $V_{j}$. The total $k$-coalition number ${\rm TC}_{k}(G)$ of $G$ equals the maximum cardinality of a total $k$-coalition partition of $G$. In this paper, the above-mentioned concepts are investigated from combinatorial points of view. Several sharp lower and upper bounds on ${\rm TC}_{k}(G)$ are proved, where the main emphasis is given on the invariant when $k=2$. As a consequence, the exact values of ${\rm TC}_2(G)$ when $G$ is a cubic graph or a $4$-regular graph are obtained. By using similar methods, an open question posed by Henning and Mojdeh regarding double coalition is answered. Moreover, ${\rm TC}_3(G)$ is determined when $G$ is a cubic graph.
\end{abstract}

 \nocite{*}
 
\section{Introduction} 
\label{sec:intro}

Coalition in graphs was introduced by~\cite{hhhmm}, and was studied in a number of subsequent papers~\cite{agk,bhp,hhhmm1, hhhmm2,hhhmm3, hhhmm4}. While the concept of coalition in graphs arises from (standard) domination in graphs, several authors studied variants of coalition that are related to other domination-type concepts. For instance, total coalition~\cite{abg,babl,hj}, independent coalition~\cite{abgk-2025, sm}, paired coalition~\cite{smn}, connected coalition~\cite{abgk} and $k$-coalition~\cite{JAB} correspond to total, independent, paired, connected and $k$-domination in graphs, respectively. In addition, transversal coalition in hypergraphs was introduced recently~\cite{hy} presenting a coalition version of transversals in hypergraphs. 

In this paper, we present a coalition counterpart to total $k$-domination~\cite{bhs,hk,kaz,Kulli}. The latter concept is also known under the names $k$-tuple total domination~\cite{hk,hy0} and total $k$-tuple domination~\cite{bbgms}. Total $k$-domination was studied from various perspectives, while the main focus was on the parameter when $k=2$. An interplay between strong transversals in hypergraphs and total $2$-domination served as a tool for obtaining sharp upper bounds on the total $2$-domination number in general graphs and in cubic graphs~\cite{hy0,HYbook}.

In Section~\ref{sec:prelim}, we establish notation and provide main definitions used in the paper. We also prove that a total $k$-coalition partition exists for all graphs with minimum degree at least $k$.
In Section~\ref{sec:bounds}, several sharp bounds on the total $k$-coalition number are proved where our emphasis is given on $k=2$. We prove that $\delta(G)-k+2\le \TC_{k}(G)\leq n(G)-k+1$ holds for any graph $G$ with minimum degree at least $k$, where both bounds are sharp, and we characterize the graphs attaining the upper bound.
One of our main results is the bound $\TC_{2}(G)\leq \lfloor\frac{\delta}{2}\rfloor(\Delta-2\lfloor\frac{\delta}{2}\rfloor+1)+\lceil\frac{\delta}{2}\rceil,$ which holds for all graphs $G$ with minimum degree $\delta\ge 2$ and maximum degree $\Delta\ge 4\lfloor\frac{\delta}{2}\rfloor-2$, and we construct a family of graphs attaining the bound for every even minimum degree $\delta\ge 2$.
In addition, this enables us to make use of some techniques that are more effective in relation to small values of $k$. In particular, as a consequence of this approach, we give the exact values of $\TC_{2}$ in the case of cubic graphs and $4$-regular graphs, and determine $\TC_{3}(G)$ for any cubic graph $G$.

\cite{HM} studied double coalition in graphs, which can be considered as a closed variant of total $2$-coalition. Indeed, if in the definition of total $2$-coalition one replaces open neighborhoods with closed neighborhoods, we get the definition of double coalition. The corresponding invariant of $G$ is the {\em double coalition number}, denoted $\DC(G)$. It was proved in~\cite{HM} that $\DC(G)=4$ for any cubic graph $G$. Based on this fact and some other pieces of evidence, Henning and Mojdeh asked if $\DC(G)\leq 1+\Delta(G)$ holds for any graph $G$ with $\delta(G)=3$. Some of the methods developed in Section~\ref{sec:bounds} can be applied to this question. In fact, we give a negative answer to it by presenting an infinite family of graphs $G$ of minimum degree 3 for which $\DC(G)$ is arbitrarily greater than $1+\Delta(G)$. In addition, we provide an upper bound for $\DC(G)$, which is expressed as a function of $\delta(G)$ and $\Delta(G)$, and is sharp for arbitrarily large $\delta(G)$ and $\Delta(G)$; the bound is proved for $\Delta(G)\ge  4\lceil\frac{\delta(G)}{2}\rceil-3$.


\section{Preliminaries}
\label{sec:prelim}

Throughout this paper, we consider $G$ as a finite, connected and simple graph with vertex set $V(G)$ and edge set $E(G)$. We use~\cite{West} as a reference for terminology and notation which are not explicitly defined here. The ({\em open}) {\em neighborhood} of a vertex $v$ is denoted by $N_{G}(v)$, and its {\em closed neighborhood} is $N_{G}[v]=N_{G}(v)\cup \{v\}$. When $G$ will be clear from the context, we may simplify the notation to $N(v)$ and $N[v]$. The {\em minimum} and {\em maximum degrees} of $G$ are denoted by $\delta(G)$ and $\Delta(G)$, respectively.

A set of vertices $S\subseteq V(G)$ is a \textit{dominating set} (resp.\ \textit{total dominating set}) if every vertex in $V(G)\setminus S$ \big{(}resp. $V(G)$\big{)} has a neighbor is $S$.

The study of a natural generalization of total domination was initiated by~\cite{Kulli} as follows. For $k\ge 1$ and a graph $G$ of minimum degree at least $k$, a subset $S\subseteq V(G)$ is a \textit{total $k$-dominating set} if $|N(v)\cap S|\geq k$ for all $v\in V(G)$. This concept was later investigated by~\cite{hk} under the name \textit{$k$-tuple total domination}. A vertex partition of such a graph into total $k$-dominating sets is called the \textit{total $k$-domatic partition} of $G$. (Since $V(G)$ is a total $k$-dominating set of a graph $G$ with minimum degree at least $k$, such a partition exists in $G$.) The \textit{total $k$-domatic number} of $G$, denoted by $d_{\times k,t}(G)$, is the maximum cardinality taken over all total $k$-domatic partitions of $G$~\cite{SL}.

A \textit{total $k$-coalition} in a graph $G$ with $\delta(G)\geq k$ consists of two disjoint sets $U,V\subseteq V(G)$, such that neither $U$ nor $V$ is a total $k$-dominating set, but the union $U\cup V$ is a total $k$-dominating set in $G$. We say that $V$ is a {\em partner} of $U$ (and $U$ is a partner of $V$). A \textit{total $k$-coalition partition} in $G$ is a vertex partition $\Omega=\{V_{1},\ldots,V_{|\Omega|}\}$ such that every set $V_{i}$ forms a total $k$-coalition with another set $V_{j}$. The \textit{total $k$-coalition number} $\TC_{k}(G)$ equals the maximum cardinality taken over all total $k$-coalition partitions in $G$.\footnote{One of the referees informed us that the total $k$-coalition in graphs for $k=2$  was already studied under the name \textit{double total coalition} in~\cite{gab}. The special cases of the present Propositions~\ref{Existence} and~\ref{Large}, and of Lemma~\ref{L1} for $k=2$ can be found in~\cite{gab}.}

\cite{hh} initiated the study of tuple domination, which is conceptually close to total $k$-domination. In fact, for a graph $G$ with $\delta(G)\geq k-1$, a {\em $k$-tuple dominating set} is defined by making use of closed neighborhood instead of open neighborhood in the definition of a total $k$-dominating set. Note that a $2$-tuple dominating set is usually referred to as a \textit{double dominating set}. In view of this, similarly to the concept of total $2$-coalition, a \textit{double coalition partition} can be defined for all graphs with minimum degree at least $1$. \cite{HM} investigated double coalition in graphs. 

By an $\eta(G)$-partition, where $\eta\in\{\TC, \DC, d_{\times k,t}\}$, we mean an $\eta$-partition of $G$ of largest cardinality.

First, we show that total $k$-coalition partitions exist for all graphs of minimum degree at least $k$.

\begin{proposition}\label{Existence}
Any graph $G$ of minimum degree at least $k$ has a total $k$-coalition partition.
\end{proposition}
\begin{proof}
Let $\Omega=\{V_{1},\ldots,V_{|\Omega|}\}$ be a $d_{\times k,t}(G)$-partition. Then $|\Omega|=d_{\times k,t}(G)$. We may assume that $V_{1},\ldots,V_{|\Omega|-1}$ are minimal total $k$-dominating sets. Otherwise, we replace them with minimal total $k$-dominating sets $V_{1}'\subseteq V_{1},\ldots,V_{|\Omega|-1}'\subseteq V_{|\Omega|-1}$ respectively, and replace $V_{|\Omega|}$ with $V_{|\Omega|}\cup\big{(}\cup_{i=1}^{|\Omega|-1}(V_{i}\setminus V_{i}')\big{)}$. Let $\{V_{i,1},V_{i,2}\}$ be any partition of $V_{i}$ for each $i\in[|\Omega|-1]$. It is then clear, by definitions, that $V_{i,1}$ and $V_{i,2}$ form a total $k$-coalition in $G$. If $V_{|\Omega|}$ turns out to be a minimal total $k$-dominating set, then $\Theta=\{V_{i,1},V_{i,2}\}_{i=1}^{|\Omega|}$ is a total $k$-coalition partition in $G$, in which $\{V_{|\Omega|,1},V_{|\Omega|,2}\}$ is any partition of $V_{|\Omega|}$. Otherwise, we replace $V_{|\Omega|}$ with a minimal total $k$-dominating set $V_{|\Omega|}'\subseteq V_{|\Omega|}$ and set $V_{|\Omega|}''=V_{|\Omega|}\setminus V_{|\Omega|}'$. Notice that $V_{|\Omega|}''$ is not a total $k$-dominating set in $G$ as $\Omega$ is a $d_{\times k,t}(G)$-partition. Let $\{V_{|\Omega|,1}',V_{|\Omega|,2}'\}$ be any partition of $V_{|\Omega|}'$. If $V_{|\Omega|}''$ forms a total $k$-coalition with $V_{|\Omega|,1}'$ or $V_{|\Omega|,2}'$, then $\{V_{i,1},V_{i,2}\}_{i=1}^{|\Omega|-1}\cup \{V_{|\Omega|,1}',V_{|\Omega|,2}',V_{|\Omega|}''\}$ will be a total $k$-coalition partition. So, we assume that neither $V_{|\Omega|,1}'$ nor $V_{|\Omega|,2}'$ forms a total $k$-coalition with $V_{|\Omega|}''$. In such a case, $V_{|\Omega|,1}'\cup V_{|\Omega|}''$ is a total $k$-coalition partner of $V_{|\Omega|,2}'$. Therefore, $\{V_{i,1},V_{i,2}\}_{i=1}^{|\Omega|-1}\cup \{V_{|\Omega|,1}'\cup V_{|\Omega|}'',V_{|\Omega|,2}'\}$ is a desired partition.
\end{proof}

Invoking the proof of Proposition \ref{Existence}, we deduce that $\TC_{k}(G)\geq2d_{\times k,t}(G)$ for any graph $G$ with $\delta(G)\geq k$.


\section{Total $k$-coalition with an emphasis on $k=2$}\label{sec:bounds}

First, we present general lower and upper bounds on the total $k$-coalition number of a graph in terms of minimum and maximum degrees, respectively. Then, with emphasis on $k=2$, we give two upper bounds on $\TC_{2}(G)$ in terms of both minimum and maximum degrees, and show that they are sharp by exhibiting an infinite family of graphs, which is illustrated in Example \ref{Exam}. 

\begin{theorem}\label{lower}
For any graph $G$ with minimum degree at least $k$, $\TC_{k}(G)\geq \delta(G)-k+2$. This bound is sharp.
\end{theorem}
\begin{proof}
Let $v\in V(G)$ be a vertex of minimum degree and let $N(v)=\{v_{1},\ldots,v_{\delta(G)}\}$. We set $V'=V(G)\setminus \{v_{1},\ldots,v_{\delta(G)-k+1}\}$ and $V_{i}=\{v_{i}\}$ for $i\in [\delta(G)-k+1]$. Obviously, no set $V_{i}$ is a total $k$-dominating set in $G$. Moreover, since $v$ has precisely $k-1$ neighbors in $V'$, it follows that $V'$ is not a total $k$-dominating set in $G$ either. Let $V_{i}$ be any set where $i\in [\delta(G)-k+1]$ and $u$ be any vertex in $G$. On the other hand, because $|N(u)\cap \big(V(G)\setminus(V'\cup V_{i})\big)|\leq \delta(G)-k\leq \deg_G(u)-k$, it follows that $u$ is adjacent to at least $k$ vertices in $V'\cup V_{i}$. Hence, $V_{i}$ and $V'$ form a total $k$-coalition in $G$ for each $i\in [\delta(G)-k+1]$. The above discussion shows that $\Omega=\big{\{}V_{1},\ldots,V_{\delta(G)-k+1},V'\big{\}}$ is a total $k$-coalition partition in $G$. Thus, $\TC_{k}(G)\geq|\Omega|=\delta(G)-k+2$.

To see that the bound is sharp,  consider the cycle $C_{n}$ and the complete graph $K_{n}$, for which we have $\TC_{k}(C_{n})=2$ (for $k=2$) and $\TC_{k}(K_{n})=n-k+1$ (for each positive integer $k$ with $n\geq k+1$).
\end{proof}

Note that no two disjoint subsets $A,B\subseteq V(G)$ with $|A\cup B|\leq k$ form a total $k$-coalition in any graph $G$ with $\delta(G)\geq k$. This leads to the clear upper bound $\TC_{k}(G)\leq|V(G)|-k+1$. However, an infinite family of nontrivial graphs attain this upper bound. Recall that $G\vee H$ denotes the join of graphs $G$ and $H$.

\begin{proposition}\label{Large}
Let $G$ be a graph of order $n$ with $\delta(G)\geq k\geq2$. Then, $\TC_{k}(G)\leq n-k+1$ holds with equality if and only if $G\cong K_{k}\vee G'$, where $G'$ is any graph of order $n-k$.
\end{proposition}
\begin{proof}
Suppose first that $\TC_{k}(G)=n-k+1$ and that $\Omega=\{V_{1},\ldots,V_{|\Omega|}\}$ is a $\TC_{k}(G)$-partition. Then $|\Omega|=n-k+1$. Letting $V_{1}$ form a total $k$-coalition with $V_{2}$, we get 
\begin{center}
$n=|V_{1}\cup V_{2}|+\sum_{i=3}^{|\Omega|}|V_{i}|\geq 
(k+1)+|\Omega|-2$.
\end{center}
Due to this, the equality $|\Omega|=n-k+1$ guarantees that
\begin{enumerate}
\item[(i)] $|V_{1}\cup V_{2}|=k+1$ and $|V_{i}|=1$ for each $i\in[|\Omega|]\setminus \{1,2\}$,
\item[(ii)] $G[V_{1}\cup V_{2}]\cong K_{k+1}$, and
\item[(iii)] every singleton set $|V_{i}|$, for $i\in[|\Omega|]\setminus \{1,2\}$, forms a total $k$-coalition with $V_{1}$ or $V_{2}$.
\end{enumerate}

By taking the above statements into account, without loss of generality, we may assume that $|V_{1}|=k$ and $|V_{2}|=1$. Therefore, $\Omega=\{V_{1}\}\cup \big{\{}\{v\}\mid v\in V(G)\setminus V_{1}\}\big{\}}$, in which $\{v\}$ forms a total $k$-coalition with $V_{1}$ for each $v\in V(G)\setminus V_{1}$. In particular, $vx\in E(G)$ for all vertices $x\in V_{1}$ and $v\in V(G)\setminus V_{1}$. Therefore, $G\cong K_{k}\vee G'$, where $G'=G[V(G)\setminus V_{1}]$.

Conversely, assume that $G\cong K_{k}\vee G'$, in which $G'$ is any graph of order $n-k$. It is then easy to see that $\{V(K_{k})\}\cup \big{\{}\{v\}\mid v\in V(G')\big{\}}$ is a total $k$-coalition partition in $G$ of cardinality $n-k+1$, and hence $\TC_{k}(G)\geq n-k+1$. This leads to the desired equality due to the upper bound $\TC_{k}(G)\leq n-k+1$.
\end{proof}

The following lemma will turn out to be useful in several places of this paper.  

\begin{lemma}\label{L1}
Let $G$ be a graph with minimum degree at least $k$ and let $\Omega$ be a $\TC_{k}(G)$-partition. If $A\in \Omega$, then $A$ forms a total $k$-coalition with at most $\Delta(G)-k+1$ sets in $\Omega$.
\end{lemma}
\begin{proof}
Since $A$ is not a total $k$-dominating set in $G$, there exists a vertex $v$ such that $|N(v)\cap A|\leq k-1$. Let $A$ form a total $k$-coalition with $A_{1},\ldots,A_{t}\in \Omega$. By definition and since $A$ does not totally $k$-dominate $v$, it follows that $v$ has at least $k$ neighbors in $A\cup A_{t}$ and $|N(v)\cap A_{i}|\geq1$ for each $i\in [t]$. Then  
$$\Delta(G)\geq|N(v)|\geq|N(v)\cap A_{1}|+\cdots+|N(v)\cap A_{t-1}|+|N(v)\cap (A\cup A_{t})|\geq t-1+k\,,$$
which proves the result. 
\end{proof}

Associated with any total $k$-coalition partition $\Omega=\{V_{1},\ldots,V_{|\Omega|}\}$ in a graph $G$ with $\delta(G)\geq k$, the \textit{total $k$-coalition graph} $\TC_k\mathrm{G}(G,\Omega)$ has the set of vertices $\Omega$, in which two vertices $V_{i}$ and $V_{j}$ are adjacent if they form a total $k$-coalition in $G$. Recall that $\alpha(G)$ and $\beta(G)$ denote the independence number and the vertex cover number of $G$, respectively.

\begin{lemma}\label{covering}
Let $G$ be a graph of minimum degree at least $2$. If $\Omega$ is a total $2$-coalition partition of $G$, then the following statements hold.
\begin{enumerate}
\item[\emph{(}i\emph{)}] $\Delta\big{(}\TC_{2}G(G,\Omega)\big{)}\leq \Delta(G)-1$, and
\item[\emph{(}ii\emph{)}] $\beta\big{(}\TC_{2}G(G,\Omega)\big{)}\leq \delta(G)-1$.
\end{enumerate}
\end{lemma}
\begin{proof}
$(i)$ The statement follows from Lemma~\ref{L1} with $k=2$.

$(ii)$ Let $v$ be a vertex of minimum degree in $G$. We set $\Omega'=\{A\in \Omega\mid N(v)\cap A\neq \emptyset\}$. Obviously, $|\Omega'|\leq \delta(G)$. Suppose that $|\Omega'|=\delta(G)$. This shows that every set in $\Omega'$ contains precisely one vertex from $N(v)$. In such a situation, $v$ has at most one neighbor in $A\cup(\cup_{S\in \Omega\setminus \Omega'}S)$ for each $A\in \Omega'$. This shows that no two sets in $\mathcal{I}_A=\{A\}\cup \{S\mid S\in \Omega\setminus \Omega'\}$ form a total $2$-coalition in $G$, in which $A$ is any set in $\Omega'$. Equivalently, $\mathcal{I}_A$ is an independent set in $\TC_{2}G(G,\Omega)$, and hence $\alpha\big{(}\TC_{2}G(G,\Omega)\big{)}\geq|\mathcal{I}_A|=|\Omega|-\delta(G)+1$. Using the equality $\alpha(H)+\beta(H)=|V(H)|$, which holds for each graph $H$ (the Gallai Theorem), we infer that $\beta\big{(}\TC_{2}G(G,\Omega)\big{)}\leq \delta(G)-1$. Moreover, if $|\Omega'|<\delta(G)$, then no two sets in $\{S\mid S\in \Omega\setminus \Omega'\}$ form a total $2$-coalition in $G$. In a similar fashion, the inequality $\beta\big{(}\TC_{2}G(G,\Omega)\big{)}\leq \delta(G)-1$ is obtained. 
\end{proof}

Apart from bounding the total $2$-coalition number of a graph, the following result together with Theorem \ref{lower} will enable us to obtain the exact value of this parameter for cubic graphs and for $4$-regular graphs.

\begin{theorem}\label{Combination}
If $G$ is a graph with $\delta(G)\ge 2$, then 
\begin{center}
$\TC_{2}(G)\leq \max\Big\{\Delta(G),\left\lfloor \frac{\delta(G)}{2}\right\rfloor(\Delta(G)-4)+\delta(G)\Big\}$.
\end{center}
Moreover, this bound is sharp.
\end{theorem}
\begin{proof}
Let $\Omega$ be a $\TC_{2}(G)$-partition. If $\delta(G)=2$, then $\beta\big{(}\TC_{2}G(G,\Omega)\big{)}=1$ by Lemma \ref{covering}$(ii)$. Therefore, $\TC_{2}G(G,\Omega)$ has a universal vertex $V$, that is, $V\in \Omega$ forms a total $2$-coalition with any other set in $\Omega$. So, we have $\TC_{2}(G)\leq \Delta(G)$ in view of Lemma \ref{L1}.  

Now let $\delta(G)=3$ and let $v$ (resp.\ $u$) be a vertex of maximum (resp.\ minimum) degree in $G$. We are going to show that $\TC_{2}(G)\le \Delta(G)$ also in this case. Suppose that $\TC_{2}(G)>\Delta(G)$ and that $\Omega=\{V_{1},\ldots,V_{|\Omega|}\}$ is a $\TC_{2}(G)$-partition. We distinguish two cases depending on the behavior of the sets in $\Omega$.

\medskip\noindent
\textbf{Case 1.} $|N(v)\cap V_{i}|\leq 1$, $i\in[|\Omega|]$. \\
Since no two vertices in $N(v)$ belong to the same set in $\Omega$, we may assume that $N(v)\subseteq V_{1}\cup \cdots \cup V_{\Delta(G)}$. Note that $V_{\Delta(G)+1}$ forms a total $2$-coalition with $V_{j}$ for some $j\in[|\Omega|]$. Since $|N(v)\cap V_{j}|\leq1$ and because $|N(v)\cap(V_{\Delta(G)+1}\cup V_{j})|\geq2$, it follows that $v$ has at least one neighbor in $V_{\Delta(G)+1}$, which contradicts the fact that $\deg_G(v)=\Delta(G)$.

\medskip\noindent
\textbf{Case 2.} There are at least two vertices in $N(v)$ that belong to the same set in $\Omega$. \\
In such a situation, we assume without loss of generality that $N(v)\subseteq V_{1}\cup \cdots \cup V_{p}$ for some $p\in[\Delta(G)-1]$ and that $V_{i}$ has at least one vertex in $N(v)$ for each $i\in[p]$. 
We may assume without loss of generality that $|N(v)\cap V_{1}|\geq2,\ldots,|N(v)\cap V_{t}|\geq2$ for some $t\in[p]$ and that $|N(v)\cap V_{i}|\leq1$ for the remaining sets $V_{i}$ in $\Omega$. The following claim will turn out to be useful.

\medskip\noindent
\textbf{Claim 1.} \textit{For each index $i>t$, the set $V_{i}$ forms a total $2$-coalition with a set $V_{j}$ for some $j\in[t]$.}

\medskip\noindent
\textit{Proof of Claim 1}. We first consider an arbitrary index $i>p$. Suppose $V_{i}$ forms a total $2$-coalition with $V_{j}$ for some $j>t$. Then, the resulting inequality $|N(v)\cap(V_{i}\cup V_{j})|\geq2$ and the fact that $|N(v)\cap V_{j}|\leq1$ imply that $v$ has at least one neighbor in $V_{i}$, in contradiction with the equality $\deg_G(v)=\Delta(G)$. Therefore, $V_{i}$ is a total $2$-coalition partner of $V_{j}$ for some $j\in[t]$. In particular, we may assume that $V_{\Delta(G)}$ forms a total $2$-coalition with $V_{1}$. 

We now consider any index $i\in[p]\setminus[t]$. Note that $V_{i}$ does not form a total $2$-coalition with any set $V_{j}$, with $j>p$, as proved above. Suppose that $V_{i}$ is a total $2$-coalition partner of $V_{j}$ for some $j\in[p]\setminus[t]$. This in particular implies that $|N(u)\cap(V_{i}\cup V_{j})|\geq2$. On the other hand, $|N(u)\cap(V_{1}\cup V_{\Delta(G)})|\geq2$ as $V_{1}$ and $V_{\Delta(G)}$ form a total $2$-coalition in $G$. Therefore, $3=|N(u)|\geq|N(u)\cap(V_{i}\cup V_{j})|+|N(u)\cap(V_{1}\cup V_{\Delta(G)})|\geq4$, a contradiction. Thus, $V_{i}$ is a total $2$-coalition partner of $V_{j}$ for some $j\in[t]$, proving the claim. $(\square)$

\medskip
Invoking Claim 1, we assume without loss of generality that $V_{t+1}$ forms a total $2$-coalition with $V_{1}$. On the other hand, there exists a set $V_{j}\in \Omega$ that does not form a total $2$-coalition with $V_{1}$, because $V_{1}$ is a total $2$-coalition partner of at most $\Delta(G)-1$ sets in $\Omega$ due to Lemma \ref{L1}. We need to consider two more possibilities.

\medskip\noindent
\textbf{Subcase 2.1.} $V_{j}$ forms a total $2$-coalition with $V_{r}$ for some $r\in[|\Omega|]\setminus \{1,t+1\}$.\\
In this case we have $3=|N(u)|\geq|N(u)\cap(V_{1}\cup V_{t+1})|+|N(u)\cap(V_{j}\cup V_{r})|\geq4$, which is impossible.

\medskip\noindent
\textbf{Subcase 2.2.} $V_{j}$ forms a total $2$-coalition with $V_{t+1}$. \\
If $V_{1}$ is a total $2$-coalition partner of a set $V_{r}$ for some $r\in[|\Omega|]\setminus \{t+1\}$, then $3=|N(u)|\geq|N(u)\cap(V_{1}\cup V_{r})|+|N(u)\cap(V_{j}\cup V_{t+1})|\geq4$, a contradiction. Therefore, we infer from the above argument that every set in $\Omega\setminus \{V_{1}\}$ forms a total $2$-coalition with $V_{t+1}$. Hence, $V_{t+1}$ forms a total $2$-coalition with at least $\Delta(G)$ sets in $\Omega$, contradicting Lemma \ref{L1} with $k=2$.

\medskip
By the above, we have proved that 
\begin{equation}\label{small}
\TC_{2}(G)\leq \Delta(G)
\end{equation}
when $\delta(G)\in \{2,3\}$.

\medskip\noindent
\textbf{Claim 2.} \textit{$\TC_{2}(G)=4$ for any $4$-regular graph $G$}.

\medskip\noindent
\textit{Proof of Claim 2}. Note that we already have $\TC_{2}(G)\geq4$ by Theorem \ref{lower}. Let $u\in V(G)$ and let $\Omega=\{V_{1},\ldots,V_{|\Omega|}\}$ be a $\TC_{2}(G)$-partition. Suppose that $|N(u)\cap V_{i}|\leq1$ for each $i\in[|\Omega|]$. So, we may assume without loss of generality that $N(u)\subseteq \cup_{i=1}^{4}V_{i}$ and that $|N(u)\cap V_{i}|=1$ for each $i\in[4]$. If $|\Omega|\geq5$, then let $V_{5}$ form a total $2$-coalition with $V_{j}$ for some $j\in[|\Omega|]$. Since $|N(u)\cap V_{j}|\leq1$ and $V_{5}\cup V_{j}$ is a total $2$-dominating set in $G$, it follows that $u$ has at least one neighbor in $V_{5}$, contradicting the fact that $\deg_{G}(u)=4$. Thus, $\TC_{2}(G)=|\Omega|\leq4$.

Assume that there exists exactly one set $V_{i}\in \Omega$ such that $|N(u)\cap V_{i}|\geq2$. We let, without loss of generality, $i=1$. If $|N(u)\cap V_{1}|\in \{3,4\}$, then every set in $\Omega\setminus \{V_{1}\}$ forms a total $2$-coalition with $V_{1}$. Together with Lemma~\ref{L1} for $k=2$, this shows that $\TC_{2}(G)\leq4$. Now let $|N(u)\cap V_{1}|=2$. Due to this, we can assume that $|N(u)\cap V_{2}|=|N(u)\cap V_{3}|=1$. Similarly, we deduce that every set in $\Omega\setminus \{V_{1},V_{2},V_{3}\}$ is necessarily a total $2$-coalition partner of $V_{1}$ only. If $V_{1}$ forms a total $2$-coalition with $V_{2}$ and $V_{3}$, respectively, then $\TC_{2}(G)\leq4$ by Lemma \ref{L1}. So, we assume that at least one of $V_{2}$ and $V_{3}$, say $V_{2}$, does not form a total $2$-coalition with $V_{1}$. This implies that there exists a vertex $x$ which is not totally $2$-dominated by $V_{1}\cup V_{2}$, and that $V_{2}$ forms a total $2$-coalition with $V_{3}$.

Suppose that $\TC_{2}(G)\geq5$. Since $V_{1}$ forms a total $2$-coalition with $V_{4}$ and $V_{5}$, respectively, we have  $|N(x)\cap(V_{1}\cup V_{4})|\geq2$ and $|N(x)\cap(V_{1}\cup V_{5})|\geq2$. Moreover, we have $|N(x)\cap(V_{2}\cup V_{3})|\geq2$ as $V_{2}$ forms a total $2$-coalition with $V_{3}$. Since $\deg_{G}(x)=4$, it necessarily follows that $|N(x)\cap(V_{1}\cup V_{4})|=|N(x)\cap(V_{1}\cup V_{5})|=|N(x)\cap(V_{2}\cup V_{3})|=2$. Therefore, $|N(x)\cap V_{1}|=2$, in contradiction with the fact that $x$ is not totally $2$-dominated by $V_{1}\cup V_{2}$.

Now suppose that $|N(x)\cap V_{i}|\geq2$ for at least two sets $V_{i}\in \Omega$. Without loss of generality, we let $|N(x)\cap V_{1}|=|N(x)\cap V_{2}|=2$. If one of the sets $V_{1}$ and $V_{2}$ forms a total $2$-coalition with all other sets in $\Omega$, then in view of Lemma~\ref{L1}, we have $\TC_2(G)=4$. Suppose to the contrary that $\TC_2(G)>4$. We may thus assume without loss of generality that $V_1$ forms a total $2$-coalition with $V_3$ and $V_4$, and $V_2$ forms a total $2$-coalition with $V_5$. Since $V_1$ is not a total $2$-dominating set, there is a vertex $v\in V(G)$ such that $|N(v)\cap V_1|\le 1$. If $N(v)\cap V_1=\emptyset$, then $|N(v)\cap V_3|\ge 2$, and $|N(v)\cap V_4|\ge 2$. However, since $\deg_G(v)=4$, there exists no vertex in $N(v)$ that belongs to $V_2\cup V_5$, which is a contradiction to the fact that $V_2$ and $V_5$ form a total $2$-coalition in $G$. The second possibility is that $|N(v)\cap V_1|=1$. Similarly as in the previous case, there is a vertex in $N(v)$ that belongs to $V_3$ and a vertex in $N(v)$ that belongs to $V_4$. Since $\deg_{G}(v)=4$, we derive $|N(v)\cap(V_2\cup V_5)|\le 1$, again a contradiction. This completes the proof of Claim 2. $(\square)$\vspace{1mm}

In view of \eqref{small} and Claim $2$, for graphs $G$ with $\delta(G)\in \{2,3\}$ the desired upper bound holds when $\Delta(G)\leq4$. (Note that by invoking Claim~$2$ and \eqref{small}, $\Delta(G)=4$ leads to $\TC_{2}(G)=\Delta(G)$ or $\TC_{2}(G)\leq \Delta(G)$ when $\delta(G)=4$ or $\delta(G)\in \{2,3\}$, respectively. Moreover, if $\Delta(G)\leq3$, then the resulting inclusion of $\delta(G)$ in $\{2,3\}$ leads to $\TC_{2}(G)\leq \Delta(G)$ by~\eqref{small}.) Assume in the rest that $\Delta(G)\geq5$. Let $u$ be a vertex of minimum degree in $G$. If $|N(u)\cap V_i|\le 1$ for each $i\in[|\Omega|]$, then $\TC_{2}(G)\leq \delta(G)$. Indeed, if there exists a set $V_i\in \Omega$ such that $N(u)\cap V_i=\emptyset$, then $V_i$ does not form a total $2$-coalition with any set in $\Omega$, a contradiction. So, $\TC_2(G)$ is less than or equal to the desired upper bound.  

Now assume that $|N(u)\cap V_{i}|\geq2$ for some $i\in[|\Omega|]$. We may assume, without loss of generality, that $V_{1},\ldots,V_{p}$ are the sets in $\Omega$ having at least two vertices in $N(u)$. Let $n_{i}=|N(u)\cap V_{i}|$ for each $i\in[p]$. It is clear that $p\leq \lfloor \delta(G)/2\rfloor$. Setting 
\begin{center}
$\Omega_{i}=\{V\in \Omega \mid \mbox{$V$ forms a total $2$-coalition with $V_{i}$ and $N(u)\cap V=\emptyset$}\}$
\end{center} 
for each $i\in [p]$, we deduce from Lemma \ref{L1} that $|\Omega_{i}|\leq \Delta(G)-1$.

Assume that $|\Omega_{i}|=\Delta(G)-1$ for some $i\in[p]$. Since $V_{i}$ is not a total $2$-dominating set, there exists a vertex $x\in V(G)$ such that $|N(x)\cap V_{i}|\leq1$. If $N(x)\cap V_{i}=\emptyset$, then $x$ has at least two neighbors in each set $V$ in $\Omega_{i}$ as $V$ is a total $2$-coalition partner of $V_{i}$. Therefore, $\deg_{G}(x)\geq2|\Omega_{i}|=2\Delta(G)-2$, in contradiction with $\Delta(G)\geq5$. So, we have $|N(x)\cap V_{i}|=1$. In such a situation, the vertex $x$ has at least one neighbor in each set $V\in \Omega_{i}$ because $V$ forms a total $2$-coalition with $V_{i}$. In particular, this implies that $x$ has precisely one neighbor in every set in $\Omega_{i}\cup \{V_{i}\}$ and that $\deg_{G}(x)=\Delta(G)$ since $|\Omega_{i}\cup \{V_{i}\}|=\Delta(G)$. Suppose that there exists a set $U\in \Omega\setminus(\Omega_{i}\cup \{V_{i}\})$, which forms a total $2$-coalition with a set $W\in \Omega$. Notice that $|N(x)\cap W|\in\{0,1\}$ if and only if $W\notin \Omega_{i}$ or $W\in \Omega_{i}$, respectively. This shows that $U$ has at least one vertex in $N(u)$ as $U$ and $W$ form a total $2$-coalition in $G$. Therefore, $\deg_G(x)\geq|N(x)\cap V_{i}|+|N(x)\cap(\cup_{V\in \Omega_{i}}V)|+|N(x)\cap U|\geq \Delta(G)+1$, which is impossible. The above argument guarantees that $\Omega=\Omega_{i}\cup \{V_{i}\}$, and hence $\TC_{2}(G)=\Delta(G)$.

From here on, in view of the above discussion, we assume that $|\Omega_{i}|\leq \Delta(G)-2$ for each $i\in [p]$. Moreover,  note that 
\begin{center}
$\Omega=\{V\in \Omega\mid N(u)\cap V\neq \emptyset\}\cup(\bigcup_{i=1}^{p}\Omega_{i})$.
\end{center}
 
We distinguish two cases depending of the behavior of the family $\{\Omega_{i}\}_{i=1}^{p}$.

\medskip\noindent
\textbf{Case A.} $|\Omega_{i}|\leq \Delta(G)-3$ for each $i\in [p]$. \\
In such a situation, since $\delta(G)-(n_{1}+\cdots+n_{p})$ is the number of sets in $\Omega$ that have exactly one vertex in $N(u)$, we can estimate as follows: 
\begin{align*}
\TC_{2}(G) & = |\Omega| \leq p+p(\Delta(G)-3)+\delta(G)-(n_{1}+\cdots+n_{p}) \\
& \leq p(\Delta(G)-2)+\delta(G)-2p = p(\Delta(G)-4)+\delta(G) \\
& \leq \lfloor \delta(G)/2\rfloor(\Delta(G)-4)+\delta(G).
\end{align*}

\noindent
\textbf{Case B.} $|\Omega_{i}|=\Delta(G)-2$ for some $i\in [p]$. \\
Without loss of generality, we may assume that $|\Omega_{1}|=\Delta(G)-2$. Since $V_{1}$ is not a total $2$-dominating set in $G$, there exists a vertex $x$ such that $|N(x)\cap V_{1}|\leq1$. If $N(x)\cap V_{1}=\emptyset$, then the vertex $x$ has at least two neighbors in each set in $\Omega_{1}$. Hence, $\deg_{G}(x)\geq2|\Omega_{1}|\geq2\Delta(G)-4$, contradicting the fact that $\Delta(G)\geq5$. Therefore, $|N(x)\cap V_{1}|=1$. Then, $x$ has at least one neighbor in each set in $\Omega_{1}$. If there exists a set $U\in \Omega_{i}\setminus \Omega_{1}$ for some $i\in [p]$, then 
\begin{center}
$\deg_{G}(x)\geq|N(x)\cap V_{1}|+|N(x)\cap(\cup_{V\in \Omega_{1}}V)|+|N(x)\cap(U\cup V_{i})|\geq \Delta(G)+1$,
\end{center}
a contradiction. This shows that $\Omega_{i}\subseteq \Omega_{1}$ for every $i\in [p]$. Therefore,
\begin{align}
\TC_{2}(G) & = |\Omega| \leq \Delta(G)-1+p-1+\delta(G)-(n_{1}+\cdots+n_{p}) \nonumber \\
& \leq \Delta(G)+\delta(G)-p-2  \nonumber \\
& \leq \Delta(G)+\delta(G)-3\,. \label{Jeeg2}
\end{align}

By \eqref{Jeeg2}, the desired upper bound holds when $\delta(G)\leq3$. So, we assume that $\delta(G)\geq 4$. Since $\Delta(G)\geq5$, by \eqref{Jeeg2} we thus have 
\begin{center}
$\TC_{2}(G)\leq \Delta(G)+\delta(G)-3\leq \lfloor \delta(G)/2\rfloor(\Delta(G)-4)+\delta(G)$\,. 
\end{center}
This completes the proof of the desired upper bound. The sharpness of the bound is presented in Example~\ref{Exam}.
\end{proof}

Using the obtained bounds, we derive the values of the total $k$-coalition numbers of cubic graphs for all $k\ge 2$. In this case, the necessary condition $\delta(G)\geq k$ implies that $k\in \{2,3\}$.

\begin{theorem}\label{cubic}
If $G$ is a cubic graph, then
\[
\TC_{k}(G)=\begin{cases}
3 & \textrm{if}\ \ k=2,\\
2 & \textrm{if}\ \ k=3.
\end{cases}
\]
\end{theorem}
\begin{proof}
The equality $\TC_{2}(G)=3$ is an immediate consequence of Theorems \ref{lower} and \ref{Combination}. So, we turn our attention to $k=3$. The lower bound $\TC_{3}(G)\geq2$ is clear from Theorem \ref{lower}. Let $\Omega$ be a $\TC_{3}(G)$-partition. Since $G$ is a cubic graph, it follows that for any vertex $v\in V(G)$, there exists a set $A\in \Omega$ such that $|N(v)\cap A|\leq1$. Let $B$ be any set in $\Omega$ that forms a total $3$-coalition with $A$. Particularly, we have $N(v)\subseteq A\cup B$ and $|N(v)\cap B|\geq2$. Hence, $N(v)\cap C=\emptyset$ for each $C\in \Omega\setminus \{A,B\}$. Thus, if such a set $C$ existed, it would not form a total $3$-coalition with any set in $\Omega$, therefore there is no set in $\Omega\setminus \{A,B\}$. Hence, $\TC_{3}(G)=|\Omega|=2$. 
\end{proof}

A graph $G$ with $\Delta(G)=3$ is called {\em subcubic}. If $G$ is a subcubic graph with $\delta(G) \geq 2$, then $\TC_{2}(G)\in \{2,3\}$ and both options are possible. Say, $\TC_{2}(K_{2,3})=3$, while if $G$ is obtained from the cycle $C_{5}$ by adding an edge between two nonadjacent vertices, $\TC_{2}(G)=2$.
In view of these remarks, it is natural to pose the following.

\begin{p}
Characterize subcubic graphs $G$ with $\TC_{2}(G)=2$.    
\end{p}

Note that Claim 2 of the proof of Theorem~\ref{Combination} provides the following auxiliary result, which can be of independent interest. 
\begin{proposition}
If $G$ is a $4$-regular graph, then $\TC_2(G)=4$.      
\end{proposition}

In the next result, we provide a different upper bound on $\TC_2(G)$ for graphs $G$ with sufficiently large maximum degree. If $\delta(G)\ge 6$, the bound improves that of Theorem~\ref{Combination}, yet there is no restriction on $\Delta(G)$ in Theorem~\ref{Combination}. 

\begin{theorem}\label{thm:upperbound_delta}
If $G$ is a graph with $\delta=\delta(G)\ge 2$ and $\Delta=\Delta(G)\ge 4\left\lfloor\frac{\delta}{2}\right\rfloor-2$, then $$\TC_{2}(G)\leq \left\lfloor\frac{\delta}{2}\right\rfloor(\Delta-2\left\lfloor\frac{\delta}{2}\right\rfloor+1)+\left\lceil\frac{\delta}{2}\right\rceil.$$  Moreover, the bound is sharp for every even minimum degree $\delta\ge 2$.
\end{theorem}

\begin{proof}
Since by Theorem~\ref{Combination}, the upper bound holds when $\delta(G)\le 4$, we may restrict our attention to a graph $G$ with $\delta(G)\ge 5$. 
Let $\Omega=\{V_{1},\ldots,V_{|\Omega|}\}$ be a $\TC_{2}(G)$-partition, and let $u$ be a vertex of minimum degree in $G$. If $|N(u)\cap V_i|\le 1$ for all $i\in[|\Omega|]$, then we easily see that $\TC_{2}(G)\le \delta(G)$, which is in turn bounded from above by $\left\lfloor\frac{\delta}{2}\right\rfloor(\Delta-2\left\lfloor\frac{\delta}{2}\right\rfloor+1)+\left\lceil\frac{\delta}{2}\right\rceil$, and so the statement of the theorem is proved.

Thus, we may assume that there exists an integer $s\ge 1$ such that, without loss of generality, $|N(u)\cap V_i|\ge 2$ for all $V_i\in \Omega$ with $i\in[s]$, while $|N(u)\cap V_i|\le 1$ if $i>s$. Clearly, $s\le \lfloor \delta/2\rfloor$. Let $\Psi\subsetneq\Omega$\ be the set of all $V_j$ such that $V_j\cap N(u)=\emptyset$. If $\Psi=\emptyset$, then $\TC_2(G)=|\Omega|\le s+(\delta-2s)<\delta$, which is impossible. Thus, $\Psi\ne\emptyset$. Note that every $V_j\in \Psi$ forms a total $2$-coalition with some $V_i$, where $i\in[s]$. In other words, in the graph $\TC_2G(G,\Omega)$, the vertices $V_1,\ldots, V_s$ dominate all vertices in $\Psi$. Let $r$ be the smallest number of vertices in $\{V_1,\ldots, V_s\}$ that dominate all vertices of $\Psi$. By renaming the sets if necessary, let $V_1,\ldots, V_r$ dominate all vertices in $\Psi$. Clearly, $r\in[s]$. 

For each $i\in[r]$, let $\Omega_{i}$ be the set of neighbors of $V_{i}$ in the graph $\TC_2G(G,\Omega)$ that belong to $\Psi$. By our choice of $r$, we deduce that $\Omega_{i}\nsubseteq \cup_{j\in[r]\setminus \{i\}}\Omega_{j}$. Since $V_i$ is not a total $2$-dominating set of $G$, there exists a vertex $v\in V(G)$ such that $|V_i\cap N(v)|\le 1$. Assume that $|V_i\cap N(v)|=1$. (The case when $|V_i\cap N(v)|=0$ uses similar, yet simpler arguments.) Then, all sets in $\Omega_{i}$ must have a non-empty intersection with $N(v)$. In addition, for every set $V_j\in\{V_1,\ldots,V_r\}\setminus\{V_i\}$, there exists a set $V_{j'}\in \Omega_{j}\setminus \cup_{t\in[r]\setminus \{j\}}\Omega_{t}$. Thus, there are at least two vertices in $N(v)$ that belong to $V_j\cup V_{j'}$. Altogether, we infer that $|\Omega_{i}|+2r-1\le \deg_G(v)\le \Delta$. Thus, for all $i\in[r]$, we have
\begin{equation}
\label{eq:t_i}
|\Omega_{i}|\le \Delta-2r+1.
\end{equation}

Since $\{V_1,\ldots,V_r\}$ dominates $\Psi$ in $\TC_2G(G,\Omega)$, we deduce that \begin{equation}
\label{eq:TC2}
\TC_2(G)=|\Omega|\le s+\sum_{i=1}^r{|\Omega_{i}|}+\delta-2s=\sum_{i=1}^r{|\Omega_{i}|}+\delta-s.
\end{equation}
Combining~\eqref{eq:t_i} and~\eqref{eq:TC2} we infer that
\begin{equation}
\label{eq:TC2again}
\TC_2(G)\le r(\Delta-2r+1)+\delta-s.    
\end{equation}

Note that $r\le s\le \lfloor\delta/2\rfloor$, and so the upper bound in~\eqref{eq:TC2again} is in turn bounded from above as follows:
\begin{equation}
\label{eq:TC2againagain}
r(\Delta-2r+1)+\delta-s\le r(\Delta-2r+1)+\delta-r=f(r).
\end{equation}

If $r=\lfloor\delta/2\rfloor$, then we get the desired upper bound. On the other hand, since $\Delta\geq4\lfloor\delta/2\rfloor-2$, it follows that $f$ is an increasing function on $[1,\lfloor\delta/2\rfloor-1]$. Therefore, 
\begin{align*}
\TC_{2}(G) & \leq f(r)\leq f(\left\lfloor \frac{\delta}{2}\right\rfloor-1)=f(\left\lfloor \frac{\delta}{2}\right\rfloor)+4\left\lfloor \frac{\delta}{2}\right\rfloor-2-\Delta\leq f(\left\lfloor \frac{\delta}{2}\right\rfloor) \\
& = \left\lfloor\frac{\delta}{2}\right\rfloor(\Delta-2\left\lfloor\frac{\delta}{2}\right\rfloor+1)+\left\lceil\frac{\delta}{2}\right\rceil\,,
\end{align*}
as desired. 

The sharpness of this upper bound is illustrated in Example \ref{Exam}. 
\end{proof}

\begin{exa}(Sharpness of the bounds in Theorems \ref{Combination} and \ref{thm:upperbound_delta})\label{Exam}
To see that the upper bounds given in Theorems \ref{Combination} and \ref{thm:upperbound_delta} are sharp, we introduce the graphs $G(d,\ell)$, where $d\ge 2$ and $\ell\geq2d-1$ are integers, as follows. For each $i\in [d]$, consider the set of vertices $$A_i=\{x_{i,j}:\, j\in [\ell]\},$$ 
and join a new vertex $y_i$ to each $x_{i,j}$ so that $A_i\cup\{y_i\}$ induces a star $K_{1,\ell}$. For all $i\in [d-1]$, add to the graph the edge $x_{i,1}x_{i+1,1}$ if $i$ is odd, and the edge $x_{i,2}x_{i+1,2}$ if $i$ is even. Next, for each $i\in [d]$ and $j\in [\ell]$, take a copy of the complete tripartite graph $K_{d,d,d}$, and denote it by $B_{i,j}$. Choose $S_{i,j}\subset V(B_{i,j})$ with $|S_{i,j}|=2d-1$ such that $B_{i,j}-S_{i,j}$ contains an independent set with $d$ vertices. For all $i\in [d]$ and $j\in [\ell]$, join $x_{i,j}$ to the vertices in $S_{i,j}$.   
Similarly, for each $i\in [d]$ take a copy of $K_{d,d,d}$, denote it by $B_i$, and choose $S_i\subset V(B_i)$ with $|S_i|=2d-1$ such that $B_{i}-S_{i}$ contains an independent set with $d$ vertices. For all $i\in [d]$, join $y_i$ to the vertices in $S_i$. The resulting graph is connected, and we denote it by $G(d,\ell)$. The graph $G(5,\ell)$, for $\ell\geq9$, is depicted in Fig.~\ref{Fig1}.

\begin{figure}[ht!]
\centering
\begin{tikzpicture}[scale=0.28, transform shape]
\node [draw, shape=circle, fill=black] (a11) at (-2.5,3) {};
\node [draw, shape=circle, fill=black] (a12) at (-2.5,4) {};
\node [draw, shape=circle, fill=black] (a13) at (-2.5,5) {};
\node [draw, shape=circle, fill=black] (a14) at (-2.5,6) {};
\node [draw, shape=circle, fill=black] (a15) at (-2.5,7) {};
\draw[dashed] (-2.5,5) ellipse (0.5cm and 3cm);

\node [draw, shape=circle, fill=black] (b11) at (-4.5,3) {};
\node [draw, shape=circle, fill=black] (b12) at (-4.5,4) {};
\node [draw, shape=circle, fill=black] (b13) at (-4.5,5) {};
\node [draw, shape=circle, fill=black] (b14) at (-4.5,6) {};
\node [draw, shape=circle, fill=black] (b15) at (-4.5,7) {};
\draw[dashed] (-4.5,5) ellipse (0.5cm and 3cm);

\node [draw, shape=circle, fill=black] (c11) at (-6.5,3) {};
\node [draw, shape=circle, fill=black] (c12) at (-6.5,4) {};
\node [draw, shape=circle, fill=black] (c13) at (-6.5,5) {};
\node [draw, shape=circle, fill=black] (c14) at (-6.5,6) {};
\node [draw, shape=circle, fill=black] (c15) at (-6.5,7) {};
\draw[dashed] (-6.5,5) ellipse (0.5cm and 3cm);

\draw[line width=1mm] (-2.77,7.5) -- (-4.23,7.5);
\draw[line width=1mm] (-4.77,7.5) -- (-6.23,7.5);
\draw[line width=1mm] (-2.5,8) .. controls (-2.5,9) and (-6.5,9) .. (-6.5,8);


\node [draw, shape=circle, fill=black] (g11) at (-13.5,3) {};
\node [draw, shape=circle, fill=black] (g12) at (-13.5,4) {};
\node [draw, shape=circle, fill=black] (g13) at (-13.5,5) {};
\node [draw, shape=circle, fill=black] (g14) at (-13.5,6) {};
\node [draw, shape=circle, fill=black] (g15) at (-13.5,7) {};
\draw[dashed] (-13.5,5) ellipse (0.5cm and 3cm);

\node [draw, shape=circle, fill=black] (h11) at (-15.5,3) {};
\node [draw, shape=circle, fill=black] (h12) at (-15.5,4) {};
\node [draw, shape=circle, fill=black] (h13) at (-15.5,5) {};
\node [draw, shape=circle, fill=black] (h14) at (-15.5,6) {};
\node [draw, shape=circle, fill=black] (h15) at (-15.5,7) {};
\draw[dashed] (-15.5,5) ellipse (0.5cm and 3cm);

\node [draw, shape=circle, fill=black] (i11) at (-17.5,3) {};
\node [draw, shape=circle, fill=black] (i12) at (-17.5,4) {};
\node [draw, shape=circle, fill=black] (i13) at (-17.5,5) {};
\node [draw, shape=circle, fill=black] (i14) at (-17.5,6) {};
\node [draw, shape=circle, fill=black] (i15) at (-17.5,7) {};
\draw[dashed] (-17.5,5) ellipse (0.5cm and 3cm);

\draw[line width=1mm] (-13.77,7.5) -- (-15.23,7.5);
\draw[line width=1mm] (-15.77,7.5) -- (-17.23,7.5);
\draw[line width=1mm] (-13.5,8) .. controls (-13.5,9) and (-17.5,9) .. (-17.5,8);


\node [draw, shape=circle, fill=black] (a) at (-4.5,0) {};
\draw (a12)--(a)--(a11);
\draw (a13)--(a)--(a14);
\draw (a)--(a15);
\draw (c12)--(a)--(c11);
\draw (c13)--(a)--(c14);


\node [draw, shape=circle, fill=black] (c) at (-15.5,0) {};
\draw (g12)--(c)--(g11);
\draw (g13)--(c)--(g14);
\draw (c)--(g15);
\draw (i12)--(c)--(i11);
\draw (i13)--(c)--(i14);


\node [draw, shape=circle, fill=black] (d) at (-10,-3) {};
\draw (a)--(d);
\draw (c)--(d);


\node [draw, shape=circle, fill=black] (j11) at (-8,-11) {};
\node [draw, shape=circle, fill=black] (j12) at (-8,-10) {};
\node [draw, shape=circle, fill=black] (j13) at (-8,-9) {};
\node [draw, shape=circle, fill=black] (j14) at (-8,-8) {};
\node [draw, shape=circle, fill=black] (j15) at (-8,-7) {};
\draw[dashed] (-8,-9) ellipse (0.5cm and 3cm);

\node [draw, shape=circle, fill=black] (k11) at (-10,-11) {};
\node [draw, shape=circle, fill=black] (k12) at (-10,-10) {};
\node [draw, shape=circle, fill=black] (k13) at (-10,-9) {};
\node [draw, shape=circle, fill=black] (k14) at (-10,-8) {};
\node [draw, shape=circle, fill=black] (k15) at (-10,-7) {};
\draw[dashed] (-10,-9) ellipse (0.5cm and 3cm);

\node [draw, shape=circle, fill=black] (l11) at (-12,-11) {};
\node [draw, shape=circle, fill=black] (l12) at (-12,-10) {};
\node [draw, shape=circle, fill=black] (l13) at (-12,-9) {};
\node [draw, shape=circle, fill=black] (l14) at (-12,-8) {};
\node [draw, shape=circle, fill=black] (l15) at (-12,-7) {};
\draw[dashed] (-12,-9) ellipse (0.5cm and 3cm);

\draw[line width=1mm] (-8.27,-11.5) -- (-9.73,-11.5);
\draw[line width=1mm] (-10.27,-11.5) -- (-11.73,-11.5);
\draw[line width=1mm] (-8,-12) .. controls (-8,-13) and (-12,-13) .. (-12,-12);

\draw (j12)--(d)--(j11);
\draw (j13)--(d)--(j14);
\draw (d)--(j15);
\draw (l12)--(d)--(l13);
\draw (l14)--(d)--(l15);


\node [draw, shape=circle, fill=black] (m11) at (23.5,3) {};
\node [draw, shape=circle, fill=black] (m12) at (23.5,4) {};
\node [draw, shape=circle, fill=black] (m13) at (23.5,5) {};
\node [draw, shape=circle, fill=black] (m14) at (23.5,6) {};
\node [draw, shape=circle, fill=black] (m15) at (23.5,7) {};
\draw[dashed] (23.5,5) ellipse (0.5cm and 3cm);

\node [draw, shape=circle, fill=black] (n11) at (21.5,3) {};
\node [draw, shape=circle, fill=black] (n12) at (21.5,4) {};
\node [draw, shape=circle, fill=black] (n13) at (21.5,5) {};
\node [draw, shape=circle, fill=black] (n14) at (21.5,6) {};
\node [draw, shape=circle, fill=black] (n15) at (21.5,7) {};
\draw[dashed] (21.5,5) ellipse (0.5cm and 3cm);

\node [draw, shape=circle, fill=black] (o11) at (19.5,3) {};
\node [draw, shape=circle, fill=black] (o12) at (19.5,4) {};
\node [draw, shape=circle, fill=black] (o13) at (19.5,5) {};
\node [draw, shape=circle, fill=black] (o14) at (19.5,6) {};
\node [draw, shape=circle, fill=black] (o15) at (19.5,7) {};
\draw[dashed] (19.5,5) ellipse (0.5cm and 3cm);

\draw[line width=1mm] (23.23,7.5) -- (21.77,7.5);
\draw[line width=1mm] (21.23,7.5) -- (19.77,7.5);
\draw[line width=1mm] (23.5,8) .. controls (23.5,9) and (19.5,9) .. (19.5,8);
 

\node [draw, shape=circle, fill=black] (s11) at (12.5,3) {};
\node [draw, shape=circle, fill=black] (s12) at (12.5,4) {};
\node [draw, shape=circle, fill=black] (s13) at (12.5,5) {};
\node [draw, shape=circle, fill=black] (s14) at (12.5,6) {};
\node [draw, shape=circle, fill=black] (s15) at (12.5,7) {};
\draw[dashed] (12.5,5) ellipse (0.5cm and 3cm);

\node [draw, shape=circle, fill=black] (t11) at (10.5,3) {};
\node [draw, shape=circle, fill=black] (t12) at (10.5,4) {};
\node [draw, shape=circle, fill=black] (t13) at (10.5,5) {};
\node [draw, shape=circle, fill=black] (t14) at (10.5,6) {};
\node [draw, shape=circle, fill=black] (t15) at (10.5,7) {};
\draw[dashed] (10.5,5) ellipse (0.5cm and 3cm);

\node [draw, shape=circle, fill=black] (u11) at (8.5,3) {};
\node [draw, shape=circle, fill=black] (u12) at (8.5,4) {};
\node [draw, shape=circle, fill=black] (u13) at (8.5,5) {};
\node [draw, shape=circle, fill=black] (u14) at (8.5,6) {};
\node [draw, shape=circle, fill=black] (u15) at (8.5,7) {};
\draw[dashed] (8.5,5) ellipse (0.5cm and 3cm);

\draw[line width=1mm] (12.23,7.5) -- (10.77,7.5);
\draw[line width=1mm] (10.23,7.5) -- (8.77,7.5);
\draw[line width=1mm] (12.5,8) .. controls (12.5,9) and (8.5,9) .. (8.5,8);


\node [draw, shape=circle, fill=black] (e) at (21.5,0) {};
\draw (m12)--(e)--(m11);
\draw (m13)--(e)--(m14);
\draw (e)--(m15);
\draw (o12)--(e)--(o11);
\draw (o13)--(e)--(o14);


\node [draw, shape=circle, fill=black] (g) at (10.5,0) {};
\draw (s12)--(g)--(s11);
\draw (s13)--(g)--(s14);
\draw (g)--(s15);
\draw (u12)--(g)--(u11);
\draw (u13)--(g)--(u14);


\node [draw, shape=circle, fill=black] (h) at (16,-3) {};
\draw (e)--(h);
\draw (g)--(h);


\node [draw, shape=circle, fill=black] (v11) at (18,-11) {};
\node [draw, shape=circle, fill=black] (v12) at (18,-10) {};
\node [draw, shape=circle, fill=black] (v13) at (18,-9) {};
\node [draw, shape=circle, fill=black] (v14) at (18,-8) {};
\node [draw, shape=circle, fill=black] (v15) at (18,-7) {};
\draw[dashed] (18,-9) ellipse (0.5cm and 3cm);

\node [draw, shape=circle, fill=black] (w11) at (16,-11) {};
\node [draw, shape=circle, fill=black] (w12) at (16,-10) {};
\node [draw, shape=circle, fill=black] (w13) at (16,-9) {};
\node [draw, shape=circle, fill=black] (w14) at (16,-8) {};
\node [draw, shape=circle, fill=black] (w15) at (16,-7) {};
\draw[dashed] (16,-9) ellipse (0.5cm and 3cm);

\node [draw, shape=circle, fill=black] (x11) at (14,-11) {};
\node [draw, shape=circle, fill=black] (x12) at (14,-10) {};
\node [draw, shape=circle, fill=black] (x13) at (14,-9) {};
\node [draw, shape=circle, fill=black] (x14) at (14,-8) {};
\node [draw, shape=circle, fill=black] (x15) at (14,-7) {};
\draw[dashed] (14,-9) ellipse (0.5cm and 3cm);

\draw (v11)--(h)--(v12);
\draw (v13)--(h)--(v14);
\draw (h)--(v15);

\draw (x12)--(h)--(x13);
\draw (x14)--(h)--(x15);

\draw[line width=1mm] (17.71,-11.5) -- (16.3,-11.5);
\draw[line width=1mm] (15.71,-11.5) -- (14.3,-11.5);
\draw[line width=1mm] (18,-12) .. controls (18,-13) and (14,-13) .. (14,-12);


\node [draw, shape=circle, fill=black] (y11) at (-2.5,-42) {};
\node [draw, shape=circle, fill=black] (y12) at (-2.5,-41) {};
\node [draw, shape=circle, fill=black] (y13) at (-2.5,-40) {};
\node [draw, shape=circle, fill=black] (y14) at (-2.5,-39) {};
\node [draw, shape=circle, fill=black] (y15) at (-2.5,-38) {};
\draw[dashed] (-2.5,-40) ellipse (0.5cm and 3cm);

\node [draw, shape=circle, fill=black] (z11) at (-4.5,-42) {};
\node [draw, shape=circle, fill=black] (z12) at (-4.5,-41) {};
\node [draw, shape=circle, fill=black] (z13) at (-4.5,-40) {};
\node [draw, shape=circle, fill=black] (z14) at (-4.5,-39) {};
\node [draw, shape=circle, fill=black] (z15) at (-4.5,-38) {};
\draw[dashed] (-4.5,-40) ellipse (0.5cm and 3cm);

\node [draw, shape=circle, fill=black] (a'11) at (-6.5,-42) {};
\node [draw, shape=circle, fill=black] (a'12) at (-6.5,-41) {};
\node [draw, shape=circle, fill=black] (a'13) at (-6.5,-40) {};
\node [draw, shape=circle, fill=black] (a'14) at (-6.5,-39) {};
\node [draw, shape=circle, fill=black] (a'15) at (-6.5,-38) {};
\draw[dashed] (-6.5,-40) ellipse (0.5cm and 3cm);

\draw[line width=1mm] (-2.77,-42.5) -- (-4.23,-42.5);
\draw[line width=1mm] (-4.77,-42.5) -- (-6.23,-42.5);
\draw[line width=1mm] (-2.5,-43) .. controls (-2.5,-44) and (-6.5,-44) .. (-6.5,-43);


\node [draw, shape=circle, fill=black] (e'11) at (-13.5,-42) {};
\node [draw, shape=circle, fill=black] (e'12) at (-13.5,-41) {};
\node [draw, shape=circle, fill=black] (e'13) at (-13.5,-40) {};
\node [draw, shape=circle, fill=black] (e'14) at (-13.5,-39) {};
\node [draw, shape=circle, fill=black] (e'15) at (-13.5,-38) {};
\draw[dashed] (-13.5,-40) ellipse (0.5cm and 3cm);

\node [draw, shape=circle, fill=black] (f'11) at (-15.5,-42) {};
\node [draw, shape=circle, fill=black] (f'12) at (-15.5,-41) {};
\node [draw, shape=circle, fill=black] (f'13) at (-15.5,-40) {};
\node [draw, shape=circle, fill=black] (f'14) at (-15.5,-39) {};
\node [draw, shape=circle, fill=black] (f'15) at (-15.5,-38) {};
\draw[dashed] (-15.5,-40) ellipse (0.5cm and 3cm);

\node [draw, shape=circle, fill=black] (g'11) at (-17.5,-42) {};
\node [draw, shape=circle, fill=black] (g'12) at (-17.5,-41) {};
\node [draw, shape=circle, fill=black] (g'13) at (-17.5,-40) {};
\node [draw, shape=circle, fill=black] (g'14) at (-17.5,-39) {};
\node [draw, shape=circle, fill=black] (g'15) at (-17.5,-38) {};
\draw[dashed] (-17.5,-40) ellipse (0.5cm and 3cm);

\draw[line width=1mm] (-13.77,-42.5) -- (-15.23,-42.5);
\draw[line width=1mm] (-15.77,-42.5) -- (-17.23,-42.5);
\draw[line width=1mm] (-13.5,-43) .. controls (-13.5,-44) and (-17.5,-44) .. (-17.5,-43);


\node [draw, shape=circle, fill=black] (h'11) at (-8,-28) {};
\node [draw, shape=circle, fill=black] (h'12) at (-8,-27) {};
\node [draw, shape=circle, fill=black] (h'13) at (-8,-26) {};
\node [draw, shape=circle, fill=black] (h'14) at (-8,-25) {};
\node [draw, shape=circle, fill=black] (h'15) at (-8,-24) {};
\draw[dashed] (-8,-26) ellipse (0.5cm and 3cm);

\node [draw, shape=circle, fill=black] (i'11) at (-10,-28) {};
\node [draw, shape=circle, fill=black] (i'12) at (-10,-27) {};
\node [draw, shape=circle, fill=black] (i'13) at (-10,-26) {};
\node [draw, shape=circle, fill=black] (i'14) at (-10,-25) {};
\node [draw, shape=circle, fill=black] (i'15) at (-10,-24) {};
\draw[dashed] (-10,-26) ellipse (0.5cm and 3cm);

\node [draw, shape=circle, fill=black] (j'11) at (-12,-28) {};
\node [draw, shape=circle, fill=black] (j'12) at (-12,-27) {};
\node [draw, shape=circle, fill=black] (j'13) at (-12,-26) {};
\node [draw, shape=circle, fill=black] (j'14) at (-12,-25) {};
\node [draw, shape=circle, fill=black] (j'15) at (-12,-24) {};
\draw[dashed] (-12,-26) ellipse (0.5cm and 3cm);

\draw[line width=1mm] (-8.27,-23.5) -- (-9.73,-23.5);
\draw[line width=1mm] (-10.27,-23.5) -- (-11.73,-23.5);
\draw[line width=1mm] (-8,-23) .. controls (-8,-22) and (-12,-22) .. (-12,-23);

\node [draw, shape=circle, fill=black] (i) at (-15.5,-35) {};
\node [draw, shape=circle, fill=black] (k) at (-4.5,-35) {};

\draw (y12)--(k)--(y11);
\draw (y13)--(k)--(y14);
\draw (k)--(y15);
\draw (a'12)--(k)--(a'13);
\draw (a'14)--(k)--(a'15);

\draw (e'12)--(i)--(e'11);
\draw (e'13)--(i)--(e'14);
\draw (i)--(e'15);
\draw (g'12)--(i)--(g'13);
\draw (g'14)--(i)--(g'15);

\node [draw, shape=circle, fill=black] (l) at (-10,-32) {};
\draw (k)--(l);
\draw (i)--(l);

\draw (h'12)--(l)--(h'11);
\draw (h'13)--(l)--(h'14);
\draw (l)--(h'15);
\draw (j'11)--(l)--(j'12);
\draw (j'13)--(l)--(j'14);


\node [draw, shape=circle, fill=black] (h'11) at (18,-28) {};
\node [draw, shape=circle, fill=black] (h'12) at (18,-27) {};
\node [draw, shape=circle, fill=black] (h'13) at (18,-26) {};
\node [draw, shape=circle, fill=black] (h'14) at (18,-25) {};
\node [draw, shape=circle, fill=black] (h'15) at (18,-24) {};
\draw[dashed] (18,-26) ellipse (0.5cm and 3cm);

\node [draw, shape=circle, fill=black] (i'11) at (16,-28) {};
\node [draw, shape=circle, fill=black] (i'12) at (16,-27) {};
\node [draw, shape=circle, fill=black] (i'13) at (16,-26) {};
\node [draw, shape=circle, fill=black] (i'14) at (16,-25) {};
\node [draw, shape=circle, fill=black] (i'15) at (16,-24) {};
\draw[dashed] (16,-26) ellipse (0.5cm and 3cm);

\node [draw, shape=circle, fill=black] (j'11) at (14,-28) {};
\node [draw, shape=circle, fill=black] (j'12) at (14,-27) {};
\node [draw, shape=circle, fill=black] (j'13) at (14,-26) {};
\node [draw, shape=circle, fill=black] (j'14) at (14,-25) {};
\node [draw, shape=circle, fill=black] (j'15) at (14,-24) {};
\draw[dashed] (14,-26) ellipse (0.5cm and 3cm);

\draw[line width=1mm] (17.73,-23.5) -- (16.27,-23.5);
\draw[line width=1mm] (15.73,-23.5) -- (14.27,-23.5);
\draw[line width=1mm] (18,-23) .. controls (18,-22) and (14,-22) .. (14,-23);

\node [draw, shape=circle, fill=black] (m) at (10.5,-35) {};
\node [draw, shape=circle, fill=black] (o) at (21.5,-35) {};


\node [draw, shape=circle, fill=black] (k'11) at (12.5,-42) {};
\node [draw, shape=circle, fill=black] (k'12) at (12.5,-41) {};
\node [draw, shape=circle, fill=black] (k'13) at (12.5,-40) {};
\node [draw, shape=circle, fill=black] (k'14) at (12.5,-39) {};
\node [draw, shape=circle, fill=black] (k'15) at (12.5,-38) {};
\draw[dashed] (12.5,-40) ellipse (0.5cm and 3cm);

\node [draw, shape=circle, fill=black] (l'11) at (10.5,-42) {};
\node [draw, shape=circle, fill=black] (l'12) at (10.5,-41) {};
\node [draw, shape=circle, fill=black] (l'13) at (10.5,-40) {};
\node [draw, shape=circle, fill=black] (l'14) at (10.5,-39) {};
\node [draw, shape=circle, fill=black] (l'15) at (10.5,-38) {};
\draw[dashed] (10.5,-40) ellipse (0.5cm and 3cm);

\node [draw, shape=circle, fill=black] (m'11) at (8.5,-42) {};
\node [draw, shape=circle, fill=black] (m'12) at (8.5,-41) {};
\node [draw, shape=circle, fill=black] (m'13) at (8.5,-40) {};
\node [draw, shape=circle, fill=black] (m'14) at (8.5,-39) {};
\node [draw, shape=circle, fill=black] (m'15) at (8.5,-38) {};
\draw[dashed] (8.5,-40) ellipse (0.5cm and 3cm);

\draw[line width=1mm] (12.23,-42.5) -- (10.77,-42.5);
\draw[line width=1mm] (10.23,-42.5) -- (8.77,-42.5);
\draw[line width=1mm] (12.5,-43) .. controls (12.5,-44) and (8.5,-44) .. (8.5,-43);


\node [draw, shape=circle, fill=black] (n'11) at (23.5,-42) {};
\node [draw, shape=circle, fill=black] (n'12) at (23.5,-41) {};
\node [draw, shape=circle, fill=black] (n'13) at (23.5,-40) {};
\node [draw, shape=circle, fill=black] (n'14) at (23.5,-39) {};
\node [draw, shape=circle, fill=black] (n'15) at (23.5,-38) {};
\draw[dashed] (23.5,-40) ellipse (0.5cm and 3cm);

\node [draw, shape=circle, fill=black] (o'11) at (21.5,-42) {};
\node [draw, shape=circle, fill=black] (o'12) at (21.5,-41) {};
\node [draw, shape=circle, fill=black] (o'13) at (21.5,-40) {};
\node [draw, shape=circle, fill=black] (o'14) at (21.5,-39) {};
\node [draw, shape=circle, fill=black] (o'15) at (21.5,-38) {};
\draw[dashed] (21.5,-40) ellipse (0.5cm and 3cm);

\node [draw, shape=circle, fill=black] (p'11) at (19.5,-42) {};
\node [draw, shape=circle, fill=black] (p'12) at (19.5,-41) {};
\node [draw, shape=circle, fill=black] (p'13) at (19.5,-40) {};
\node [draw, shape=circle, fill=black] (p'14) at (19.5,-39) {};
\node [draw, shape=circle, fill=black] (p'15) at (19.5,-38) {};
\draw[dashed] (19.5,-40) ellipse (0.5cm and 3cm);

\draw[line width=1mm] (23.23,-42.5) -- (21.77,-42.5);
\draw[line width=1mm] (21.23,-42.5) -- (19.77,-42.5);
\draw[line width=1mm] (23.5,-43) .. controls (23.5,-44) and (19.5,-44) .. (19.5,-43);

\node [draw, shape=circle, fill=black] (p) at (16,-32) {};
\draw (m)--(p);
\draw (o)--(p);

\draw (h'12)--(p)--(h'11);
\draw (h'13)--(p)--(h'14);
\draw (p)--(h'15);
\draw (j'11)--(p)--(j'12);
\draw (j'13)--(p)--(j'14);

\draw (k'12)--(m)--(k'11);
\draw (k'13)--(m)--(k'14);
\draw (m)--(k'15);
\draw (m'12)--(m)--(m'13);
\draw (m'14)--(m)--(m'15);

\draw (n'12)--(o)--(n'11);
\draw (n'13)--(o)--(n'14);
\draw (o)--(n'15);
\draw (p'12)--(o)--(p'13);
\draw (p'14)--(o)--(p'15);


\draw (a)--(g);
\draw (k)--(m);
\draw (e)--(o);


\node [draw, shape=circle, fill=black] (q'11) at (-0.5,-15) {};
\node [draw, shape=circle, fill=black] (q'12) at (-0.5,-14) {};
\node [draw, shape=circle, fill=black] (q'13) at (-0.5,-13) {};
\node [draw, shape=circle, fill=black] (q'14) at (-0.5,-12) {};
\node [draw, shape=circle, fill=black] (q'15) at (-0.5,-11) {};
\draw[dashed] (-0.5,-13) ellipse (0.5cm and 3cm);

\node [draw, shape=circle, fill=black] (r'11) at (-2.5,-15) {};
\node [draw, shape=circle, fill=black] (r'12) at (-2.5,-14) {};
\node [draw, shape=circle, fill=black] (r'13) at (-2.5,-13) {};
\node [draw, shape=circle, fill=black] (r'14) at (-2.5,-12) {};
\node [draw, shape=circle, fill=black] (r'15) at (-2.5,-11) {};
\draw[dashed] (-2.5,-13) ellipse (0.5cm and 3cm);

\node [draw, shape=circle, fill=black] (s'11) at (-4.5,-15) {};
\node [draw, shape=circle, fill=black] (s'12) at (-4.5,-14) {};
\node [draw, shape=circle, fill=black] (s'13) at (-4.5,-13) {};
\node [draw, shape=circle, fill=black] (s'14) at (-4.5,-12) {};
\node [draw, shape=circle, fill=black] (s'15) at (-4.5,-11) {};
\draw[dashed] (-4.5,-13) ellipse (0.5cm and 3cm);

\draw[line width=1mm] (-0.77,-10.5) -- (-2.23,-10.5);
\draw[line width=1mm] (-2.77,-10.5) -- (-4.23,-10.5);
\draw[line width=1mm] (-0.5,-10) .. controls (-0.5,-9) and (-4.5,-9) .. (-4.5,-10);


\node [draw, shape=circle, fill=black] (111) at (10.5,-15) {};
\node [draw, shape=circle, fill=black] (112) at (10.5,-14) {};
\node [draw, shape=circle, fill=black] (113) at (10.5,-13) {};
\node [draw, shape=circle, fill=black] (114) at (10.5,-12) {};
\node [draw, shape=circle, fill=black] (115) at (10.5,-11) {};
\draw[dashed] (10.5,-13) ellipse (0.5cm and 3cm);

\node [draw, shape=circle, fill=black] (211) at (8.5,-15) {};
\node [draw, shape=circle, fill=black] (212) at (8.5,-14) {};
\node [draw, shape=circle, fill=black] (213) at (8.5,-13) {};
\node [draw, shape=circle, fill=black] (214) at (8.5,-12) {};
\node [draw, shape=circle, fill=black] (215) at (8.5,-11) {};
\draw[dashed] (8.5,-13) ellipse (0.5cm and 3cm);

\node [draw, shape=circle, fill=black] (311) at (6.5,-15) {};
\node [draw, shape=circle, fill=black] (312) at (6.5,-14) {};
\node [draw, shape=circle, fill=black] (313) at (6.5,-13) {};
\node [draw, shape=circle, fill=black] (314) at (6.5,-12) {};
\node [draw, shape=circle, fill=black] (315) at (6.5,-11) {};
\draw[dashed] (6.5,-13) ellipse (0.5cm and 3cm);

\draw[line width=1mm] (10.23,-10.5) -- (8.77,-10.5);
\draw[line width=1mm] (8.23,-10.5) -- (6.77,-10.5);
\draw[line width=1mm] (10.5,-10) .. controls (10.5,-9) and (6.5,-9) .. (6.5,-10);


\node [draw, shape=circle, fill=black] (x'11) at (5,-30) {};
\node [draw, shape=circle, fill=black] (x'12) at (5,-29) {};
\node [draw, shape=circle, fill=black] (x'13) at (5,-28) {};
\node [draw, shape=circle, fill=black] (x'14) at (5,-27) {};
\node [draw, shape=circle, fill=black] (x'15) at (5,-26) {};
\draw[dashed] (5,-28) ellipse (0.5cm and 3cm);

\node [draw, shape=circle, fill=black] (y'11) at (3,-30) {};
\node [draw, shape=circle, fill=black] (y'12) at (3,-29) {};
\node [draw, shape=circle, fill=black] (y'13) at (3,-28) {};
\node [draw, shape=circle, fill=black] (y'14) at (3,-27) {};
\node [draw, shape=circle, fill=black] (y'15) at (3,-26) {};
\draw[dashed] (3,-28) ellipse (0.5cm and 3cm);

\node [draw, shape=circle, fill=black] (z'11) at (1,-30) {};
\node [draw, shape=circle, fill=black] (z'12) at (1,-29) {};
\node [draw, shape=circle, fill=black] (z'13) at (1,-28) {};
\node [draw, shape=circle, fill=black] (z'14) at (1,-27) {};
\node [draw, shape=circle, fill=black] (z'15) at (1,-26) {};
\draw[dashed] (1,-28) ellipse (0.5cm and 3cm);

\draw[line width=1mm] (4.73,-30.5) -- (3.27,-30.5);
\draw[line width=1mm] (2.73,-30.5) -- (1.27,-30.5);
\draw[line width=1mm] (5,-31) .. controls (5,-32) and (1,-32) .. (1,-31);


\node [draw, shape=circle, fill=black] (q) at (3,-22) {};

\draw (x'12)--(q)--(x'11);
\draw (x'13)--(q)--(x'14);
\draw (q)--(x'15);
\draw (z'12)--(q)--(z'13);
\draw (z'14)--(q)--(z'15);

\node [draw, shape=circle, fill=black] (r) at (-2.5,-19) {};
\node [draw, shape=circle, fill=black] (t) at (8.5,-19) {};

\draw (s'12)--(r)--(s'11);
\draw (s'13)--(r)--(s'14);
\draw (r)--(s'15);
\draw (q'11)--(r)--(q'12);
\draw (q'13)--(r)--(q'14);

\draw (312)--(t)--(311);
\draw (313)--(t)--(314);
\draw (t)--(315);
\draw (111)--(t)--(112);
\draw (113)--(t)--(114);

\draw (r)--(q);
\draw (t)--(q);
\draw (i) .. controls (-19,-19) and (-10,-19) .. (r);


\node [scale=1.8] at (-10.8,-3.2) {\LARGE $y_{1}$};
\node [scale=1.8] at (-4.2,-0.8) {\LARGE $x_{1,1}$};
\node [scale=1.8] at (-16,-0.6) {\LARGE $x_{1,\ell}$};
\node [scale=1.8] at (-4.25,9.5) {\LARGE $B_{1,1}$};
\node [scale=1.8] at (-15.7,9.5) {\LARGE $B_{1,\ell}$};
\node [scale=1.8] at (-10,-13.6) {\LARGE $B_{1}$};

\node [scale=1.8] at (15.1,-3.2) {\LARGE $y_{2}$};
\node [scale=1.8] at (10.2,-0.6) {\LARGE $x_{2,1}$};
\node [scale=1.8] at (22.6,-0.5) {\LARGE $x_{2,2}$};
\node [scale=1.8] at (21.45,9.5) {\LARGE $B_{2,2}$};
\node [scale=1.8] at (10.2,9.5) {\LARGE $B_{2,1}$};
\node [scale=1.8] at (15.9,-13.6) {\LARGE $B_{2}$};

\node [scale=1.8] at (15.15,-31.7) {\LARGE $y_{3}$};
\node [scale=1.8] at (11.78,-35.29) {\LARGE $x_{3,1}$};
\node [scale=1.8] at (22.68,-35.03) {\LARGE $x_{3,2}$};
\node [scale=1.8] at (10.7,-44.6) {\LARGE $B_{3,1}$};
\node [scale=1.8] at (21.7,-44.6) {\LARGE $B_{3,2}$};
\node [scale=1.8] at (15.9,-21.5) {\LARGE $B_{3}$};

\node [scale=1.8] at (-10.8,-31.7) {\LARGE $y_{4}$};
\node [scale=1.8] at (-5.7,-35.2) {\LARGE $x_{4,1}$};
\node [scale=1.8] at (-16.71,-35.1) {\LARGE $x_{4,2}$};
\node [scale=1.8] at (-15.2,-44.6) {\LARGE $B_{4,2}$};
\node [scale=1.8] at (-4.2,-44.6) {\LARGE $B_{4,1}$};
\node [scale=1.8] at (-10,-21.5) {\LARGE $B_{4}$};

\node [scale=1.8] at (3.93,-22.3) {\LARGE $y_{5}$};
\node [scale=1.8] at (-3.25,-19.6) {\LARGE $x_{5,2}$};
\node [scale=1.8] at (9.5,-19.59) {\LARGE $x_{5,\ell}$};
\node [scale=1.8] at (-2.5,-8.5) {\LARGE $B_{5,2}$};
\node [scale=1.8] at (8.5,-8.5) {\LARGE $B_{5,\ell}$};
\node [scale=1.8] at (3.1,-32.5) {\LARGE $B_{5}$};

\node [scale=1.8] at (-10,5) {\LARGE $.\ .\ .\ .$};
\node [scale=1.8] at (-10,-40) {\LARGE $.\ .\ .\ .$};
\node [scale=1.8] at (16,5) {\LARGE $.\ .\ .\ .$};
\node [scale=1.8] at (16,-40) {\LARGE $.\ .\ .\ .$};
\node [scale=1.8] at (3,-13) {\LARGE $.\ .\ .\ .$};

\end{tikzpicture}
\caption{{\small The graph $G(d,\ell)$ (for $d=5$ and $\ell\geq9$) given in Example~\ref{Exam} with $\delta\big{(}G(5,\ell)\big{)}=10$, $\Delta\big{(}G(5,\ell)\big{)}=9+\ell$, and $\TC_2\big{(}G(5,\ell)\big{)}=5(\ell+1)$. Here, each thick line/curve segment represents all possible edges between the corresponding partite sets.}}\label{Fig1}
\end{figure}

The degree in $G(d,\ell)$ of the vertices from the sets $A_i, B_{i,j}$ and $B_i$ is either $2d$ or $2d+1$. On the other hand, the degree of the vertices $y_i$ is $\ell+2d-1$. Since $\ell \ge 2$, we get $2d+1\le \ell+2d-1$. So, 
$$\delta\big{(}G(d,\ell)\big{)}=2d\quad \textrm{and}\quad \Delta\big{(}G(d,\ell)\big{)}=\ell+2d-1,$$
which satisfy $\Delta\big{(}G(d,\ell)\big{)}\geq4\lfloor \delta\big{(}G(d,\ell)\big{)}/2\rfloor-2$ as $\ell\geq2d-1$.

Now, let us present a partition $\Omega=\{V_1,\ldots,V_{|\Omega|}\}$ of $V\big{(}G(d,\ell)\big{)}$, where $|\Omega|=d(\ell+1)$, for which we will show it is a total $2$-coalition partition. For all $i\in [d]$ and $j\in [\ell]$, vertices of each partite set from $B_{i,j}$ are distributed to the sets $V_1,\ldots,V_d$, respectively. In addition, note that $S_{i,j}$ is chosen in such a way that $V(B_{i,j})\setminus S_{i,j}$ contains all vertices of one partite set of $B_{i,j}$ and an additional vertex, and we may assume that this additional vertex is in the set $V_i$. For all $i\in [d]$, we let $y_i$ belong to $V_i$. Similarly, vertices of each partite set from $B_{i}$ are distributed to the sets $V_1,\ldots,V_d$, respectively. In addition, since $S_{i}$ is chosen in such a way that $V(B_{i})\setminus S_i$ contains all vertices of one partite set of $B_{i}$ and an additional vertex, we may assume that this additional vertex is put in the set $V_i$. Finally, for each $i\in [d]$, the vertices $x_{i,1},x_{i,2},\ldots,x_{i,\ell}\in A_{i}$ are distributed into the sets 
\begin{center}
$V_{d+(i-1)\ell+1},V_{d+(i-1)\ell+2},\ldots,V_{d+i\ell}$,
\end{center} 
respectively. More precisely, $V_{d+(i-1)\ell+j}=\{x_{i,j}\}$ for all $i\in [d]$ and $j\in[\ell]$. The set $V_i$, $i\in [d]$, is not a total $2$-dominating set because $y_i$ is adjacent to exactly one vertex of $V_i$. On the other hand, $V_i$
forms a total $2$-coalition with every set in $\{V_{d+(i-1)\ell+1},V_{d+(i-1)\ell+2},\ldots,V_{d+i\ell}\}$. Hence, $\Omega$ is a total $2$-coalition partition of cardinality $d(\ell+1)$. Therefore, 
\begin{equation*}
\begin{array}{ll}
    \TC_2(G(d,\ell))&\ge d(\ell+1) \\ 
    & = d\big((\ell+2d-1)-2d+1\big)+d\\
     & =\left\lfloor\frac{\delta(G(d,\ell))}{2}\right\rfloor(\Delta(G(d,\ell))-2\left\lfloor\frac{\delta(G(d,\ell))}{2}\right\rfloor+1)+\left\lceil\frac{\delta(G(d,\ell))}{2}\right\rceil\\
     & \ge  \TC_2(G(d,\ell)).
\end{array}    
\end{equation*}
We infer that the graphs $G(d,\ell)$ attain the upper bound of Theorem \ref{thm:upperbound_delta}.

It should be noted that the bound is also sharp when $\delta=2$. In fact, the upper bound gives the exact value $\TC_2(C_n)=2$ in this case.

In view of Theorem \ref{lower} with $k=2$, the upper bound given in Theorem \ref{Combination} gives us the exact value of the total $2$-coalition number of $r$-regular graphs for $r\in \{2,3,4\}$. On the other hand, for each $\ell\geq3$, we get
\begin{center}
$\TC_{2}\big{(}G(2,\ell)\big{)}=2\ell+2=\left\lfloor \frac{\delta(G(2,\ell))}{2}\right\rfloor\big{(}\Delta\big{(}G(2,\ell)\big{)}-4\big{)}+\delta\big{(}G(2,\ell)\big{)}$.
\end{center}
Therefore, the graphs $G(2,\ell)$ attain the upper bound of Theorem \ref{Combination}.
\end{exa}

It is possible that Theorem~\ref{thm:upperbound_delta} is true even if the restriction $\Delta(G)\ge 4\lfloor \delta(G)/2\rfloor-2$ is omitted. In spite of extensive investigations, we could not find a counterexample to that statement. In addition, we base our suspicion that the restriction can be omitted because this is true in the case when $\delta(G)\le 5$, which follows from Theorem~\ref{Combination}. Based on the above discussion, we propose the following:

\begin{con}
If $G$ is a graph with $\delta=\delta(G)\ge 2$ and $\Delta=\Delta(G)$, then $$\TC_{2}(G)\leq \left\lfloor\frac{\delta}{2}\right\rfloor(\Delta-2\left\lfloor\frac{\delta}{2}\right\rfloor+1)+\left\lceil\frac{\delta}{2}\right\rceil.$$  
\end{con}
As noted earlier, if the conjecture holds, then it is widely sharp. Notably, the family of graphs $G(d,\ell)$ from Example~\ref{Exam} shows that it is sharp for an arbitrary even $\delta\ge2$ and any $\Delta\ge2\delta-2$. In addition, there are regular graphs $G$ with even $\delta(G)=\Delta(G)$ for which the bound is sharp. Consider the complete graph $K_{2p+1}$ for any integer $p\ge 1$. Note that $$\TC_2(K_{2p+1})=2p=\left\lfloor\frac{\delta}{2}\right\rfloor(\Delta-2\left\lfloor\frac{\delta}{2}\right\rfloor+1)+\left\lceil\frac{\delta}{2}\right\rceil.$$


\section{On two open problems on double coalition}

\cite{HM} proved that $\DC(G)\leq1+\Delta(G)$ for all graphs $G$ with $\delta(G)\in\{1,2\}$. They posed the following:

\medskip\noindent
\textit{Question 1. If $G$ is a graph with $\delta(G)=3$, then is it true that $\DC(G)\leq1+\Delta(G)$?}

\medskip
This is indeed the case when the graph is cubic, as they proved that $\DC(G)=4$ for each cubic graph $G$. Despite the above-mentioned pieces of evidence in support of the inequality, in what follows, we answer this question in the negative. 

Moreover, by utilizing the approach developed for total 2-domination (Theorem~\ref{thm:upperbound_delta}), we present a general upper bound on the double coalition number of a graph $G$ in terms of minimum and maximum degrees, provided that $\Delta(G)\ge 4\lceil \delta(G)/2\rceil-3$. Since the bound is sharp for all odd $\delta(G)$, where $\Delta(G)$ can be arbitrarily large, we see that the value of $\DC(G)$ can be relatively close to $\lceil \delta(G)/2\rceil \Delta(G)$. Given a graph $G$ and a double coalition partition $\Omega$, the graph $\DCG(G,\Omega)$ is defined with vertex set $\Omega$ in which two vertices/sets are adjacent if they form a double coalition. 

\begin{theorem}\label{thm:doubledomination}
If $G$ is a graph with $\delta=\delta(G)\ge1$ and $\Delta=\Delta(G)\ge 4\left\lceil\frac{\delta}{2}\right\rceil-3$, then $$\DC(G)\leq \left\lceil\frac{\delta}{2}\right\rceil(\Delta-2\left\lceil\frac{\delta}{2}\right\rceil+2)+1+\left\lfloor\frac{\delta}{2}\right\rfloor.$$  Moreover, the bound is sharp, and is attained for graphs with any odd minimum degree $\delta\ge 3$.
\end{theorem}
\begin{proof}
Let $\Omega=\{V_{1},\ldots,V_{|\Omega|}\}$ be a $\DC(G)$-partition, and let $u$ be a vertex of minimum degree in $G$. If $|N[u]\cap V_i|\le 1$ for all $i\in[|\Omega|]$, then $\DC(G)\le \delta(G)+1$, which directly implies the statement of the theorem.

Thus, we may assume that there exists an integer $s\ge 1$ such that, without loss of generality, $|N[u]\cap V_i|\ge 2$ for all $V_i\in \Omega$ with $i\in[s]$, while $|N[u]\cap V_i|\le 1$ if $i>s$. Clearly, $s\le \lceil \delta/2\rceil$. Let $\Psi\subsetneq\Omega$ be the set of all $V_j$ such that $V_j\cap N[u]=\emptyset$. If $\Psi=\emptyset$, then $\DC(G)=|\Omega|\le s+(\delta+1-2s)<\delta+1$, which is impossible. Thus, $\Psi\ne\emptyset$. Note that every $V_j\in \Psi$ forms a double coalition with some $V_i$, where $i\in[s]$. In other words, in the graph $\DCG(G,\Omega)$, the vertices $V_1,\ldots, V_s$ dominate all vertices in $\Psi$. Let $r$ be the smallest number of vertices in $\{V_1,\ldots, V_s\}$ that dominate all vertices of $\Psi$. By renaming the sets if necessary, let $V_1,\ldots, V_r$ dominate all vertices in $\Psi$. Clearly, $r\in[s]$. 

For each $i\in[r]$, let $\Omega_{i}$ be the set of neighbors of $V_{i}$ in the graph $\DCG(G,\Omega)$ that belong to $\Psi$. By our choice of $r$, we deduce that $\Omega_{i}\nsubseteq \cup_{j\in[r]\setminus \{i\}}\Omega_{j}$. Since $V_i$ is not a double dominating set of $G$, there exists a vertex $v\in V(G)$ such that $|V_i\cap N[v]|\le 1$. Assume that $|V_i\cap N[v]|=1$. (The case when $|V_i\cap N[v]|=0$ uses similar, yet slightly simpler arguments.) Then, all sets in $\Omega_{i}$ must have a non-empty intersection with $N[v]$. In addition, for every set $V_j\in\{V_1,\ldots,V_r\}\setminus\{V_i\}$, there exists a set $V_{j'}\in \Omega_{j}\setminus \cup_{t\in[r]\setminus \{j\}}\Omega_{t}$. Thus, there are at least two vertices in $N[v]$ that belong to $V_j\cup V_{j'}$. Altogether, we infer that $|\Omega_{i}|+2r-1\le \deg_G(v)+1\le \Delta+1$. Thus, for all $i\in[r]$, we have
\begin{equation}
\label{eq:t_inova}
|\Omega_{i}|\le \Delta-2r+2.
\end{equation}

Since $\{V_1,\ldots,V_r\}$ dominates $\Psi$ in $\DCG(G,\Omega)$, we deduce that \begin{equation}
\label{eq:DC}
\textrm{DC}(G)=|\Omega|\le s+\sum_{i=1}^r{|\Omega_{i}|}+\delta+1-2s=\sum_{i=1}^r{|\Omega_{i}|}+\delta+1-s.
\end{equation}
Combining \eqref{eq:t_inova} and \eqref{eq:DC} we get
\begin{equation}
\label{eq:DCagain}
\DC(G)\le r(\Delta-2r+2)+\delta+1-s.    
\end{equation}

Note that $r\le s\le \lceil\delta/2\rceil$, and so the upper bound in~\eqref{eq:DCagain} is in turn bounded from above as follows:
\begin{equation}
\label{eq:DCagainagain}
r(\Delta-2r+2)+\delta+1-s\le r(\Delta-2r+2)+\delta+1-r=f(r).
\end{equation}

If $r=\lceil\delta/2\rceil$, then we get the desired upper bound. Now let $r<\lceil\delta/2\rceil$. Since $\Delta\geq 4\lceil\delta/2\rceil-3$, it follows that $f$ is a nondecreasing function on $[1,\lceil\delta/2\rceil-1]$. Therefore, 
\begin{align*}
\DC(G) & \leq f(r)\leq f(\lceil \frac{\delta}{2}\rceil-1)=f(\lceil \frac{\delta}{2}\rceil)+4\lceil \frac{\delta}{2}\rceil-\Delta-3\leq f(\lceil \frac{\delta}{2}\rceil) \\
& = \left\lceil\frac{\delta}{2}\right\rceil(\Delta-2\left\lceil\frac{\delta}{2}\right\rceil+2)+1+\left\lfloor\frac{\delta}{2}\right\rfloor\,,
\end{align*}
as desired. 

\begin{figure}[ht!]
\centering
\begin{tikzpicture}[scale=0.40, transform shape]
\node [draw, shape=circle, fill=black] (y_{1}) at (0,0) {};
\node [draw, shape=circle, fill=black] (x_{1,1}) at (-4.5,2) {};
\node [draw, shape=circle, fill=black] (x_{1,3}) at (4.5,2) {};
\draw (x_{1,1})--(y_{1});
\draw (y_{1})--(x_{1,3});
\node [scale=1.8] at (-0.8,-0.13) {\large $y_{1}$};
\node [scale=1.8] at (4.5,1.2) {\large $x_{1,1}$};
\node [scale=1.8] at (-4.5,1.2) {\large $x_{1,t}$};
\node [scale=1.8] at (0,-4.9) {\Large $B_{1}$};
\node [scale=1.8] at (-4.5,8.95) {\Large $B_{1,t}$};
\node [scale=1.8] at (4.5,8.95) {\Large $B_{1,1}$};

\node [draw, shape=circle, fill=black] (111) at (-6,4) {};
\node [draw, shape=circle, fill=black] (112) at (-3,4) {};
\node [draw, shape=circle, fill=black] (113) at (-5.85,5) {};
\node [draw, shape=circle, fill=black] (114) at (-3.15,5) {};
\node [draw, shape=circle, fill=black] (115) at (-5.7,6) {};
\node [draw, shape=circle, fill=black] (116) at (-3.3,6) {};
\node [draw, shape=circle, fill=black] (117) at (-5.55,7) {};
\node [draw, shape=circle, fill=black] (118) at (-3.45,7) {};
\node [draw, shape=circle, fill=black] (119) at (-5.1,7.8) {};
\node [draw, shape=circle, fill=black] (110) at (-3.9,7.8) {};

\draw (111)--(x_{1,1})--(112);
\draw (113)--(x_{1,1})--(114);
\draw (115)--(x_{1,1})--(116);
\draw (117)--(x_{1,1})--(118);
\draw (119)--(x_{1,1})--(110);

\node [draw, shape=circle, fill=black] (131) at (3.1,4) {};
\node [draw, shape=circle, fill=black] (132) at (6.1,4) {};
\node [draw, shape=circle, fill=black] (133) at (3.25,5) {};
\node [draw, shape=circle, fill=black] (134) at (5.95,5) {};
\node [draw, shape=circle, fill=black] (135) at (3.4,6) {};
\node [draw, shape=circle, fill=black] (136) at (5.8,6) {};
\node [draw, shape=circle, fill=black] (137) at (3.6,7) {};
\node [draw, shape=circle, fill=black] (138) at (5.65,7) {};
\node [draw, shape=circle, fill=black] (139) at (4,7.8) {};
\node [draw, shape=circle, fill=black] (130) at (5.2,7.8) {};

\draw (131)--(x_{1,3})--(132);
\draw (133)--(x_{1,3})--(134);
\draw (135)--(x_{1,3})--(136);
\draw (137)--(x_{1,3})--(138);
\draw (139)--(x_{1,3})--(130);

\node [draw, shape=circle, fill=black] (z_{11}) at (-3,-2.5) {};
\node [draw, shape=circle, fill=black] (z_{12}) at (-1.5,-2.5) {};
\node [draw, shape=circle, fill=black] (z_{13}) at (0,-2.5) {};
\node [draw, shape=circle, fill=black] (z_{14}) at (1.5,-2.5) {};
\node [draw, shape=circle, fill=black] (z_{15}) at (3,-2.5) {};
\node [draw, shape=circle, fill=black] (z_{16}) at (-3,-3.5) {};
\node [draw, shape=circle, fill=black] (z_{17}) at (-1.5,-3.5) {};
\node [draw, shape=circle, fill=black] (z_{18}) at (0,-3.5) {};
\node [draw, shape=circle, fill=black] (z_{19}) at (1.5,-3.5) {};
\node [draw, shape=circle, fill=black] (z_{10}) at (3,-3.5) {};

\draw (z_{11})--(y_{1})--(z_{12});
\draw (z_{14})--(y_{1})--(z_{15});
\draw (z_{16})--(y_{1})--(z_{17});
\draw (z_{19})--(y_{1})--(z_{10});


\node [draw, shape=circle, fill=black] (y_{2}) at (20,0) {};
\node [draw, shape=circle, fill=black] (x_{2,1}) at (15.5,2) {};
\node [draw, shape=circle, fill=black] (x_{2,3}) at (24.5,2) {};
\draw (x_{2,1})--(y_{2});
\draw (y_{2})--(x_{2,3});
\node [scale=1.8] at (19.15,-0.13) {\large $y_{2}$};
\node [scale=1.8] at (15.5,1.2) {\large $x_{2,1}$};
\node [scale=1.8] at (25.25,1.2) {\large $x_{2,2}$};
\node [scale=1.8] at (20,-4.9) {\Large $B_{2}$};
\node [scale=1.8] at (15.5,8.95) {\Large $B_{2,1}$};
\node [scale=1.8] at (24.5,8.95) {\Large $B_{2,2}$};

\node [draw, shape=circle, fill=black] (211) at (14,4) {};
\node [draw, shape=circle, fill=black] (212) at (17,4) {};
\node [draw, shape=circle, fill=black] (213) at (14.15,5) {};
\node [draw, shape=circle, fill=black] (214) at (16.85,5) {};
\node [draw, shape=circle, fill=black] (215) at (14.35,6) {};
\node [draw, shape=circle, fill=black] (216) at (16.7,6) {};
\node [draw, shape=circle, fill=black] (217) at (14.5,7) {};
\node [draw, shape=circle, fill=black] (218) at (16.55,7) {};
\node [draw, shape=circle, fill=black] (219) at (14.85,7.8) {};
\node [draw, shape=circle, fill=black] (210) at (16.2,7.8) {};

\draw (211)--(x_{2,1})--(212);
\draw (213)--(x_{2,1})--(214);
\draw (215)--(x_{2,1})--(216);
\draw (217)--(x_{2,1})--(218);
\draw (219)--(x_{2,1})--(210);

\node [draw, shape=circle, fill=black] (231) at (23,4) {};
\node [draw, shape=circle, fill=black] (232) at (26,4) {};
\node [draw, shape=circle, fill=black] (233) at (23.15,5) {};
\node [draw, shape=circle, fill=black] (234) at (25.85,5) {};
\node [draw, shape=circle, fill=black] (235) at (23.35,6) {};
\node [draw, shape=circle, fill=black] (236) at (25.7,6) {};
\node [draw, shape=circle, fill=black] (237) at (23.5,7) {};
\node [draw, shape=circle, fill=black] (238) at (25.55,7) {};
\node [draw, shape=circle, fill=black] (239) at (23.85,7.8) {};
\node [draw, shape=circle, fill=black] (230) at (25.2,7.8) {};

\draw (231)--(x_{2,3})--(232);
\draw (233)--(x_{2,3})--(234);
\draw (235)--(x_{2,3})--(236);
\draw (237)--(x_{2,3})--(238);
\draw (239)--(x_{2,3})--(230);

\node [draw, shape=circle, fill=black] (z_{21}) at (17,-2.5) {};
\node [draw, shape=circle, fill=black] (z_{22}) at (18.5,-2.5) {};
\node [draw, shape=circle, fill=black] (z_{23}) at (20,-2.5) {};
\node [draw, shape=circle, fill=black] (z_{24}) at (21.5,-2.5) {};
\node [draw, shape=circle, fill=black] (z_{25}) at (23,-2.5) {};
\node [draw, shape=circle, fill=black] (z_{26}) at (17,-3.5) {};
\node [draw, shape=circle, fill=black] (z_{27}) at (18.5,-3.5) {};
\node [draw, shape=circle, fill=black] (z_{28}) at (20,-3.5) {};
\node [draw, shape=circle, fill=black] (z_{29}) at (21.5,-3.5) {};
\node [draw, shape=circle, fill=black] (z_{20}) at (23,-3.5) {};

\draw (z_{21})--(y_{2})--(z_{22});
\draw (z_{24})--(y_{2})--(z_{25});
\draw (z_{26})--(y_{2})--(z_{27});
\draw (z_{29})--(y_{2})--(z_{20});


\node [draw, shape=circle, fill=black] (z_{41}) at (17,-16.5) {};
\node [draw, shape=circle, fill=black] (z_{42}) at (18.5,-16.5) {};
\node [draw, shape=circle, fill=black] (z_{43}) at (20,-16.5) {};
\node [draw, shape=circle, fill=black] (z_{44}) at (21.5,-16.5) {};
\node [draw, shape=circle, fill=black] (z_{45}) at (23,-16.5) {};
\node [draw, shape=circle, fill=black] (z_{46}) at (17,-17.5) {};
\node [draw, shape=circle, fill=black] (z_{47}) at (18.5,-17.5) {};
\node [draw, shape=circle, fill=black] (z_{48}) at (20,-17.5) {};
\node [draw, shape=circle, fill=black] (z_{49}) at (21.5,-17.5) {};
\node [draw, shape=circle, fill=black] (z_{40}) at (23,-17.5) {};

\node [draw, shape=circle, fill=black] (y_{4}) at (20,-20) {};
\node [draw, shape=circle, fill=black] (x_{4,1}) at (15.5,-22) {};
\node [draw, shape=circle, fill=black] (x_{4,3}) at (24.5,-22) {};
\draw (x_{4,1})--(y_{4});
\draw (y_{4})--(x_{4,3});
\node [scale=1.8] at (25.3,-21.3) {\large $x_{3,2}$};
\node [scale=1.8] at (15.5,-21.25) {\large $x_{3,1}$};
\node [scale=1.8] at (19.1,-19.85) {\large $y_{3}$};
\node [scale=1.8] at (20,-15.1) {\Large $B_{3}$};
\node [scale=1.8] at (15.5,-29) {\Large $B_{3,1}$};
\node [scale=1.8] at (24.5,-29) {\Large $B_{3,2}$};

\draw (z_{41})--(y_{4})--(z_{42});
\draw (z_{44})--(y_{4})--(z_{45});
\draw (z_{46})--(y_{4})--(z_{47});
\draw (z_{49})--(y_{4})--(z_{40});

\node [draw, shape=circle, fill=black] (411) at (14,-24) {}; 
\node [draw, shape=circle, fill=black] (412) at (17,-24) {};
\node [draw, shape=circle, fill=black] (413) at (14.15,-25) {}; 
\node [draw, shape=circle, fill=black] (414) at (16.85,-25) {};
\node [draw, shape=circle, fill=black] (415) at (14.3,-26) {}; 
\node [draw, shape=circle, fill=black] (416) at (16.7,-26) {};
\node [draw, shape=circle, fill=black] (417) at (14.45,-27) {}; 
\node [draw, shape=circle, fill=black] (418) at (16.55,-27) {};
\node [draw, shape=circle, fill=black] (419) at (14.9,-27.8) {}; 
\node [draw, shape=circle, fill=black] (410) at (16.1,-27.8) {};

\draw (411)--(x_{4,1})--(412);
\draw (413)--(x_{4,1})--(414);
\draw (415)--(x_{4,1})--(416);
\draw (417)--(x_{4,1})--(418);
\draw (419)--(x_{4,1})--(410);

\node [draw, shape=circle, fill=black] (431) at (23,-24) {}; 
\node [draw, shape=circle, fill=black] (432) at (26,-24) {};
\node [draw, shape=circle, fill=black] (433) at (23.15,-25) {}; 
\node [draw, shape=circle, fill=black] (434) at (25.85,-25) {};
\node [draw, shape=circle, fill=black] (435) at (23.3,-26) {}; 
\node [draw, shape=circle, fill=black] (436) at (25.7,-26) {};
\node [draw, shape=circle, fill=black] (437) at (23.45,-27) {}; 
\node [draw, shape=circle, fill=black] (438) at (25.55,-27) {};
\node [draw, shape=circle, fill=black] (439) at (23.9,-27.8) {}; 
\node [draw, shape=circle, fill=black] (430) at (25.1,-27.8) {};

\draw (431)--(x_{4,3})--(432);
\draw (433)--(x_{4,3})--(434);
\draw (435)--(x_{4,3})--(436);
\draw (437)--(x_{4,3})--(438);
\draw (439)--(x_{4,3})--(430);


\node [draw, shape=circle, fill=black] (z_{31}) at (-3,-16.5) {};
\node [draw, shape=circle, fill=black] (z_{32}) at (-1.5,-16.5) {};
\node [draw, shape=circle, fill=black] (z_{33}) at (0,-16.5) {};
\node [draw, shape=circle, fill=black] (z_{34}) at (1.5,-16.5) {};
\node [draw, shape=circle, fill=black] (z_{35}) at (3,-16.5) {};
\node [draw, shape=circle, fill=black] (z_{36}) at (-3,-17.5) {};
\node [draw, shape=circle, fill=black] (z_{37}) at (-1.5,-17.5) {};
\node [draw, shape=circle, fill=black] (z_{38}) at (0,-17.5) {};
\node [draw, shape=circle, fill=black] (z_{39}) at (1.5,-17.5) {};
\node [draw, shape=circle, fill=black] (z_{30}) at (3,-17.5) {};

\node [draw, shape=circle, fill=black] (y_{3}) at (0,-20) {};
\node [draw, shape=circle, fill=black] (x_{3,1}) at (-4.5,-22) {};
\node [draw, shape=circle, fill=black] (x_{3,3}) at (4.5,-22) {};
\draw (x_{3,1})--(y_{3});
\draw (y_{3})--(x_{3,3});
\node [scale=1.8] at (-0.9,-19.8) {\large $y_{4}$};
\node [scale=1.8] at (4.5,-21.25) {\large $x_{4,1}$};
\node [scale=1.8] at (-5.25,-21.3) {\large $x_{4,2}$};
\node [scale=1.8] at (4.5,-29) {\Large $B_{4,1}$};
\node [scale=1.8] at (-4.5,-29) {\Large $B_{4,3}$};
\node [scale=1.8] at (0,-15.1) {\Large $B_{4}$};

\draw (z_{31})--(y_{3})--(z_{32});
\draw (z_{34})--(y_{3})--(z_{35});
\draw (z_{36})--(y_{3})--(z_{37});
\draw (z_{39})--(y_{3})--(z_{30});

\node [draw, shape=circle, fill=black] (311) at (-6,-24) {}; 
\node [draw, shape=circle, fill=black] (312) at (-3,-24) {};
\node [draw, shape=circle, fill=black] (313) at (-5.85,-25) {}; 
\node [draw, shape=circle, fill=black] (314) at (-3.15,-25) {};
\node [draw, shape=circle, fill=black] (315) at (-5.7,-26) {}; 
\node [draw, shape=circle, fill=black] (316) at (-3.3,-26) {};
\node [draw, shape=circle, fill=black] (317) at (-5.55,-27) {}; 
\node [draw, shape=circle, fill=black] (318) at (-3.45,-27) {};
\node [draw, shape=circle, fill=black] (319) at (-5.1,-27.8) {}; 
\node [draw, shape=circle, fill=black] (310) at (-3.9,-27.8) {};

\draw (311)--(x_{3,1})--(312);
\draw (313)--(x_{3,1})--(314);
\draw (315)--(x_{3,1})--(316);
\draw (317)--(x_{3,1})--(318);
\draw (319)--(x_{3,1})--(310);

\node [draw, shape=circle, fill=black] (331) at (3,-24) {}; 
\node [draw, shape=circle, fill=black] (332) at (6,-24) {};
\node [draw, shape=circle, fill=black] (333) at (3.15,-25) {}; 
\node [draw, shape=circle, fill=black] (334) at (5.85,-25) {};
\node [draw, shape=circle, fill=black] (335) at (3.3,-26) {}; 
\node [draw, shape=circle, fill=black] (336) at (5.7,-26) {};
\node [draw, shape=circle, fill=black] (337) at (3.45,-27) {}; 
\node [draw, shape=circle, fill=black] (338) at (5.55,-27) {};
\node [draw, shape=circle, fill=black] (339) at (3.9,-27.8) {}; 
\node [draw, shape=circle, fill=black] (330) at (5.1,-27.8) {};

\draw (331)--(x_{3,3})--(332);
\draw (333)--(x_{3,3})--(334);
\draw (335)--(x_{3,3})--(336);
\draw (337)--(x_{3,3})--(338);
\draw (339)--(x_{3,3})--(330);


\node [draw, shape=circle, fill=black] (y_{5}) at (10,-13) {};
\node [draw, shape=circle, fill=black] (x_{5,1}) at (5.5,-11) {};
\node [draw, shape=circle, fill=black] (x_{5,3}) at (14.5,-11) {};
\draw (x_{5,1})--(y_{5});
\draw (y_{5})--(x_{5,3});
\node [scale=1.8] at (5.5,-11.9) {\large $x_{5,2}$};
\node [scale=1.8] at (14.5,-11.85) {\large $x_{5,t}$};
\node [scale=1.8] at (9.15,-13.1) {\large $y_{5}$};
\node [scale=1.8] at (10,-17.9) {\Large $B_{5}$};
\node [scale=1.8] at (5.5,-4.08) {\Large $B_{5,2}$};
\node [scale=1.8] at (14.5,-4.08) {\Large $B_{5,t}$};

\node [draw, shape=circle, fill=black] (511) at (4,-9) {};
\node [draw, shape=circle, fill=black] (512) at (7,-9) {};
\node [draw, shape=circle, fill=black] (513) at (4.15,-8) {};
\node [draw, shape=circle, fill=black] (514) at (6.85,-8) {};
\node [draw, shape=circle, fill=black] (515) at (4.3,-7) {};
\node [draw, shape=circle, fill=black] (516) at (6.7,-7) {};
\node [draw, shape=circle, fill=black] (517) at (4.45,-6) {};
\node [draw, shape=circle, fill=black] (518) at (6.55,-6) {};
\node [draw, shape=circle, fill=black] (519) at (4.9,-5.2) {};
\node [draw, shape=circle, fill=black] (510) at (6.1,-5.2) {};

\draw (511)--(x_{5,1})--(512);
\draw (513)--(x_{5,1})--(514);
\draw (515)--(x_{5,1})--(516);
\draw (517)--(x_{5,1})--(518);
\draw (519)--(x_{5,1})--(510);

\node [draw, shape=circle, fill=black] (531) at (13.1,-9) {};
\node [draw, shape=circle, fill=black] (532) at (16.1,-9) {};
\node [draw, shape=circle, fill=black] (533) at (13.25,-8) {};
\node [draw, shape=circle, fill=black] (534) at (15.95,-8) {};
\node [draw, shape=circle, fill=black] (535) at (13.4,-7) {};
\node [draw, shape=circle, fill=black] (536) at (15.8,-7) {};
\node [draw, shape=circle, fill=black] (537) at (13.6,-6) {};
\node [draw, shape=circle, fill=black] (538) at (15.65,-6) {};
\node [draw, shape=circle, fill=black] (539) at (14,-5.2) {};
\node [draw, shape=circle, fill=black] (530) at (15.2,-5.2) {};

\draw (531)--(x_{5,3})--(532);
\draw (533)--(x_{5,3})--(534);
\draw (535)--(x_{5,3})--(536);
\draw (537)--(x_{5,3})--(538);
\draw (539)--(x_{5,3})--(530);

\node [draw, shape=circle, fill=black] (z_{51}) at (7,-15.5) {};
\node [draw, shape=circle, fill=black] (z_{52}) at (8.5,-15.5) {};
\node [draw, shape=circle, fill=black] (z_{53}) at (10,-15.5) {};
\node [draw, shape=circle, fill=black] (z_{54}) at (11.5,-15.5) {};
\node [draw, shape=circle, fill=black] (z_{55}) at (13,-15.5) {};
\node [draw, shape=circle, fill=black] (z_{56}) at (7,-16.5) {};
\node [draw, shape=circle, fill=black] (z_{57}) at (8.5,-16.5) {};
\node [draw, shape=circle, fill=black] (z_{58}) at (10,-16.5) {};
\node [draw, shape=circle, fill=black] (z_{59}) at (11.5,-16.5) {};
\node [draw, shape=circle, fill=black] (z_{50}) at (13,-16.5) {};

\draw (z_{51})--(y_{5})--(z_{52});
\draw (z_{54})--(y_{5})--(z_{55});
\draw (z_{56})--(y_{5})--(z_{57});
\draw (z_{59})--(y_{5})--(z_{50});


\draw (x_{1,3})--(x_{2,1});
\draw (x_{2,3})--(x_{4,3});
\draw (x_{4,1})--(x_{3,3});
\draw (x_{3,1}) .. controls (-5,-13) and (-5,-13) .. (x_{5,1}); 

\draw[dashed] (-4.5,5.6) ellipse (2.25cm and 2.7cm);
\draw[dashed] (4.6,5.6) ellipse (2.25cm and 2.7cm);
\draw[dashed] (15.5,5.6) ellipse (2.25cm and 2.7cm);
\draw[dashed] (24.5,5.6) ellipse (2.25cm and 2.7cm);
\draw[dashed] (0,-3) ellipse (4cm and 1.3cm);
\draw[dashed] (24.5,-25.5) ellipse (2.25cm and 2.75cm);
\draw[dashed] (15.5,-25.5) ellipse (2.25cm and 2.75cm);
\draw[dashed] (4.5,-25.5) ellipse (2.25cm and 2.75cm);
\draw[dashed] (-4.5,-25.5) ellipse (2.25cm and 2.75cm);
\draw[dashed] (20,-3) ellipse (4cm and 1.3cm);
\draw[dashed] (20,-17) ellipse (4cm and 1.3cm);
\draw[dashed] (0,-17) ellipse (4cm and 1.3cm);
\draw[dashed] (10,-16) ellipse (4cm and 1.3cm);

\draw[dashed] (14.6,-7.5) ellipse (2.25cm and 2.75cm);
\draw[dashed] (5.5,-7.5) ellipse (2.25cm and 2.75cm);

\node [scale=1.8] at (0,5) {\LARGE $.\ .\ .\ .$};
\node [scale=1.8] at (20,5) {\LARGE $.\ .\ .\ .$};
\node [scale=1.8] at (0,-25) {\LARGE $.\ .\ .\ .$};
\node [scale=1.8] at (20,-25) {\LARGE $.\ .\ .\ .$};
\node [scale=1.8] at (10,-8) {\LARGE $.\ .\ .\ .$};

\end{tikzpicture}
\caption{{\small The graph $H(r,t)$ \big(for $r=5$ and $t\geq9$\big) given in the proof of Theorem~\ref{thm:doubledomination} with $\delta\big{(}H(5,t)\big{)}=9$, $\Delta\big{(}H(5,t)\big{)}=8+t$, and $\DC\big{(}H(5,t)\big{)}=5(t+1)$. Vertices in each of the dashed ellipses form the clique $K_{10}$.}}
\label{Fig2}
\end{figure}

For the sharpness of this upper bound, we present the family of graphs
$H(r,t)$, where $r\ge 2$ and $t\ge 4$ are integers with $t\geq2r-1$, as follows. For each $i\in [r]$, consider the set of vertices $$A_i=\{x_{i,j}:\, j\in [t]\},$$ 
and join a new vertex $y_i$ to each $x_{i,j}$ so that $A_i\cup\{y_i\}$ induces a star $K_{1,t}$. For all $i\in [r-1]$, add to the graph the edge $x_{i,1}x_{i+1,1}$ if $i$ is odd, and the edge $x_{i,2}x_{i+1,2}$ if $i$ is even. Next, for each $i\in [r]$ and $j\in [t]$, take a copy of the complete graph $K_{2r}$, and denote them by $B_{i,j}$. Join $x_{i,j}$ with all vertices from $B_{i,j}$. 
Similarly, for each $i\in [r]$ take a copy of $K_{2r}$, denote it by $B_i$, and join $y_i$ to $2r-2$ vertices in $B_i$. The resulting graph is connected, and we denote it by $H(r,t)$. See Fig.~\ref{Fig2} depicting $H(5,t)$ for $t\geq9$.  

The degree in $H(r,t)$ of the vertices in the sets $B_{i,j}$ is $2r$, the degree of the vertices in the sets $B_i$ is either $2r-1$ or $2r$, and the degree of the vertices in the sets $A_i$ is either $2r+1$ or $2r+2$. On the other hand, the degree of the vertices $y_i$ is $2r-2+t$, which is greater than or equal to $2r+2$  as $t\geq4$. Thus, 
\begin{center}
$\delta\big{(}H(r,t)\big{)}=2r-1 \textrm{ and } \Delta\big{(}H(r,t)\big{)}=2r-2+t$.
\end{center}
Moreover, $\Delta\big{(}H(r,t)\big{)}\geq4\lceil \delta\big{(}H(r,t)\big{)}/2\rceil-3$ as $t\geq2r-1$.

Now, let us present a partition $\Omega=\{V_1,\ldots,V_{|\Omega|}\}$ of $V\big{(}H(r,t)\big{)}$, where $|\Omega|=r(t+1)$, for which we will show it is a double coalition partition. For all $i\in [r]$ and $j\in [t]$, we let $|B_{i,j}\cap V_k|=2$ for every $k\in [r]$.
For all $i\in [r]$, we let $y_i$ belong to $V_i$. Next, for every $i\in [r]$, let $|B_i\cap V_k|=2$ for all $k\in [r]$ so that the two vertices of degree $2r-1$ in each $B_i$ belong to $V_i$.
Finally, for each $i\in [r]$ and $j\in [t]$, let $V_{r+(i-1)t+j}=\{x_{i,j}\}$. No set $V_i$, $i\in [r]$, is a double dominating set, but it forms a double coalition with every set in $\{V_{r+(i-1)t+1},V_{r+(i-1)t+2},\ldots,V_{r+it}\}$. Hence, $\Omega$ is a double coalition partition of cardinality $r(t+1)$. Therefore, 
\begin{equation*}
\begin{array}{ll}
    \DC\big{(}H(r,t)\big{)}&\ge r(t+1) \\ 
    & = r\big((2r-2+t)-2r+2\big)+1+(r-1)\\
     & =\left\lceil\frac{\delta(H(r,t))}{2}\right\rceil\big(\Delta(H(r,t))-2\left\lceil\frac{\delta(H(r,t))}{2}\right\rceil+2 \big)+1+\left\lfloor\frac{\delta(H(r,t))}{2}\right\rfloor\\
     & \ge \DC\big{(}H(r,t)\big{)}.
\end{array}    
\end{equation*}
Hence, the graphs $H(r,t)$ attain the upper bound.  
\end{proof}

The graphs $H(r,t)$ attain the bound in Theorem~\ref{thm:doubledomination} for all $r\ge 2$ and all $t\ge 4$, such that $t\geq2r-1$, and $$\DC\big{(}H(r,t)\big{)}=\left\lceil\frac{\delta(H(r,t))}{2}\right\rceil\Big(\Delta(H(r,t))-2\left\lceil\frac{\delta(H(r,t))}{2}\right\rceil+2 \Big)+1+\left\lfloor\frac{\delta(H(r,t))}{2}\right\rfloor$$
holds with $\delta\big{(}H(r,t)\big{)}=2r-1 \textrm{ and } \Delta\big{(}H(r,t)\big{)}=2r-2+t$. Clearly, this gives the negative answer to the question ``if $G$ is a graph with $\delta(G)=3$, then is it true that DC$(G)\leq \Delta(G)+1?$" in~\cite{HM}. Furthermore, the difference $\DC(H(r,t))-(\Delta(H(r,t))+1)$ can be made arbitrarily large. 

We remark that Theorem~\ref{thm:doubledomination} and the family of graphs $H(r,t)$ give an incomplete, but relatively satisfying answer to the following problem posed by~\cite{HM}:\\
``A natural problem is to determine a best possible upper bound on the double coalition number of a graph $G$ in terms of its minimum degree, $\delta(G)$, and maximum degree, $\Delta(G)$ $\dots$ For sufficiently large values of $\delta$ and $\Delta$ with $\delta\leq \Delta$, it would be interesting to determine a function $f(\delta,\Delta)$ such that for every
graph $G$ with minimum degree $\delta$ and maximum degree $\Delta$, we have $\DC(G) \le f(\delta, \Delta)$ and this bound is best possible.''

\section*{Acknowledgements}
The authors are thankful to the referees for helpful comments and suggestions.

\bibliographystyle{abbrvnat}
\bibliography{coalition-dmtcs}
\label{sec:biblio}

\end{document}